\documentclass[a4paper,10pt,reqno]{amsart}
\usepackage{enumerate}
\usepackage{graphicx}
\usepackage{amsfonts}
\usepackage{graphicx}
\usepackage{amsmath}
\usepackage{amssymb}
\newcommand{\Z}{{\mathbb Z}}
\newcommand{\RR}{{\mathbb R}}
\newcommand{\R}{{\mathbb R}}

\newcommand{\N}{{\mathbb N}}
\newcommand{\NN}{{\mathbb N}}

\newcommand{\comp}{\mathrm{comp}}

\newcommand{\SC}{(\mbox{SC}_\infty)}

\newcommand{\SI}{(\mbox{SI}_\infty)}

\newcommand{\al}{{\alpha}}

\newcommand{\ph}{{\varphi}}

\newcommand{\Lp}{\Delta}
\newcommand{\ra}{{\rightarrow}}

\newcommand{\dd}{{\partial}}
\newcommand{\lra}{\longrightarrow}

\newcommand{\ab}[1]{\left( #1\right)}
\newcommand{\aV}[1]{\left\Vert#1\right\Vert}
\newcommand{\ov}[1]{#1'}
\newcommand{\ow}[1]{\widetilde{#1}}

\DeclareMathOperator{\supp}{supp}

\newtheorem{theorem}{Theorem}

\newtheorem{lemma}{Lemma}[section]
\newtheorem{coro}[lemma]{Corollary}
\newtheorem{definition}[lemma]{Definition}
\newtheorem{prop}[lemma]{Proposition}
\newtheorem*{remark}{Remark}

 \newlength\headseptemp

\newcommand{\Hmm}[1]{\leavevmode{\marginpar{\tiny%
$\hbox to 0mm{\hspace*{-0.5mm}$\leftarrow$\hss}%
\vcenter{\vrule depth 0.1mm height 0.1mm width \the\marginparwidth}%
\hbox to 0mm{\hss$\rightarrow$\hspace*{-0.5mm}}$\\\relax\raggedright #1}}}
\begin{document}
% -------------------------------------------------------------%
\title[]{Dirichlet forms and stochastic completeness of graphs and subgraphs}
% -------------------------------------------------------------%
\author[]{Matthias Keller$^1$}
\author[]{Daniel Lenz$^2$}
%\author[]{$^2$}
\address{$^1$ Mathematisches Institut, Friedrich Schiller Universit\"at Jena, D-07743 Jena, Germany, m.keller@uni-jena.de}
\address{$^2$ Mathematisches Institut, Friedrich Schiller Universit\"at Jena,
D-07743 Jena, Germany, daniel.lenz@uni-jena.de, URL: http://www.analysis-lenz.uni-jena.de/
}

% -------------------------------------------------------------%
% -------------------------------------------------------------%
\begin{abstract}  We study Laplacians on graphs and networks via regular Dirichlet forms.
We give  a sufficient geometric condition for essential selfadjointness and explicitly determine the generators of the associated semigroups  on all $\ell^p$,  $1\leq p < \infty$, in this case.
  We characterize stochastic completeness  thereby generalizing all earlier corresponding results for graph Laplacians.     Finally, we study how  stochastic completeness of a subgraph is related
  to stochastic completeness of the whole graph.
\end{abstract}
% -------------------------------------------------------------%
\date{\today} %
\maketitle
% -------------------------------------------------------------%
%\begin{center}
%\emph{}
%\end{center}
\setlength{\parindent}{0pt}

\section*{Introduction}
There is a long history to the study of the heat equation  and spectral theory on graphs and networks  (see e.g. the monographs \cite{Chu,Col} and references therein). The corresponding operators are known as  Laplacians, acoustic operators or generators of symmetric Markov processes on the graph or network. A substantial part of this literature is devoted to graphs giving   Laplacians, which are bounded on $\ell^2$. Recently, certain basic  questions concerning unbounded Laplacians have received attention.   This is the starting point for our paper. More precisely, we use the framework of regular Dirichlet forms in order to
\begin{itemize}
\item define the Laplacians  on  networks via forms (Section~$1$),
\item study  essential selfadjointness (Theorem~\ref{essential}),
\item determine the generators of the associated semigroups on $\ell^p$, $1\leq p < \infty$, under suitable conditions (Theorem~\ref{generator}),
\item characterize stochastic completeness (Theorem~\ref{main0}),
\item investigate the relationship between stochastic completeness of graphs and subgraphs (Theorem  \ref{main1b}, Theorem \ref{main1},  Theorem~\ref{main2}).
\end{itemize}
The use of Dirichlet forms allows us to deal with these questions in a rather general setting.   In particular, our results seem to extend all earlier corresponding results.  Furthermore, we hope that our results and the thorough discussion of background and context may be useful in the study of further questions as well.

\smallskip

Let us discuss these topics  in more detail: There are recent investigations of  essential selfadjointness  of corresponding Laplacians by Jorgensen \cite{Jor}, of stochastic
completeness by Dodziuk and Matthai \cite{DM}, and of both essential selfadjointness and stochastic completeness by Dodziuk \cite{Dod2}, Wojciechowski \cite{Woj} (see \cite{Woj2} as well)  and  Weber  \cite{Web}. These investigations deal with locally finite graphs and the associated operators. While  \cite{DM, Dod2} treat bounded Laplacians, \cite{Jor,Woj,Web} do neither assume a uniform bound on the vertex degree nor a modification of the measure and, accordingly, the resulting Laplacians are not necessarily bounded.
It turns out that all  the Laplacians in question are special instances of generators of regular Dirichlet forms on discrete sets. In fact, there is a one-to-one correspondence between the regular Dirichlet forms on a discrete set and graphs over this set with weights satisfying a certain summability condition. This naturally raises the question to which extent similar  results  to the ones in \cite{Dod2,DM,Jor,Web, Woj} also  hold for arbitrary regular Dirichlet forms on discrete sets.

\smallskip

Our first result, Theorem~\ref{main0},   characterizes stochastic completeness for all regular  Dirichlet forms on discrete sets. This generalizes a main result of \cite{Woj} (see \cite{Dod2,DM,Web} as well for related results and a sufficient condition for stochastic completeness), which in turn is inspired by Grigor'yan's corresponding result for manifolds \cite{Gri}.
Of course, in terms of methods our considerations concerning stochastic completeness  heavily draw on existing literature  as e.g.  Sturm's \cite{Stu} for strongly local Dirichlet forms   and Grigor'yans results \cite{Gri} on Riemannian manifolds.
A crucial difference, however, is that our Dirichlet forms are not local. In this sense our results can be understood as providing some non-local counterpart to \cite{Stu,Gri}.

\smallskip

It should  be emphasized that -- unlike the cited literature --  we do  allow for non vanishing killing terms.  In order to  make sense out  of a notion  of stochastic completeness in the presence of a killing term we actually have to extend the usual definition.  This is done by our concept of stochastic completeness at infinity $\SC$ and stochastic incompleteness at infinity $\SI$. Let us be a bit more precise: Stochastic completeness concerns loss or conservation of heat. Now, loss of heat may occur for two reasons. One reason is killing within the graph by non-vanishing killing term. The other reason is heat transport to 'infinity' or the 'boundary' in finite time. This transport to infinity may happen irrespective of presence of a killing term. It is this transport to infinity which is captured by our notion of stochastic completeness at infinity. Of course, in the case of vanishing killing term stochastic completeness and stochastic completeness at infinity agree.  Our Theorem~\ref{main0} gives a unified treatment of the situation. Note that strengthening of the killing may make the graph actually more complete at infinity as discussed in  Theorem~\ref{main1b}.

\smallskip

%In some sense this new extended notion of stochastic completeness can be considered to be our main new achievement in this part of the paper. We also give a stochastic interpretation of our concept of stochastic completeness and show that it  agrees with the usual one in absence of a killing term. We refer to  our new notion of stochastic completeness as $\SC$.

Let us also  mention  strongly related work of Feller \cite{Fel,Fel2} and of Reuter \cite{Reu} dealing with uniqueness of Markov process  on discrete sets with given weights. While these works use different methods and  seem to have been somewhat neglected in the above mentioned literature, they in fact cover parts of the abstract  results on  stochastic completeness  discussed  in \cite{Woj,Web}.
They are in some sense even  more general in that they do not assume symmetry of the Markov process.  We will discuss this more specifically after the statement of our corresponding result. However, we stress already here that a crucial part of our result is not covered by  \cite{Fel,Reu}  as we allow for both  a killing term and for arbitrary measures on our underlying set.
%As far as methods are concerned we note that our basic line of argument is based on  resolvents while
%\cite{Dod2,DM,Web, Woj}  consider the heat equation.

\smallskip

Let us emphasize that our treatment requires intrinsically more effort  than the considerations of
\cite{Dod2,DM,Web, Woj} as in our setting the Laplacians (i.e.,  generators of the Dirichlet forms) are known much less explicitly.  In fact, in general not even the functions with compact support will be in the  domain of definition of our Laplacians.

\smallskip

As the functions with compact support need not  belong to the domain of definition of our Laplacians, the question of essential selfadjointness does in general not  make sense in our context.  On the other hand, if the functions with compact support belong to the domain of definition and a certain geometric condition -- called $(A)$
below -- is satisfied, we can prove essential selfadjointness of the Laplacians in question on the set of  functions with compact support (Theorem~\ref{essential}).  This result extends the corresponding result of \cite{Jor,Dod2,Web, Woj} to all regular Dirichlet forms on discrete sets.  Note that this (again) includes the presence  of an arbitrary  killing term and an arbitrary measure on our discrete set. We also give examples in which  essential selfadjointness fails (as does  condition $(A)$).

\smallskip

Along our way, we can also  determine   the generators for the corresponding semigroups on all $\ell^p$, $p\in [1,\infty)$,  for all regular Dirichlet forms on graphs satisfying $(A)$.   These generators turn out to be the ``maximal'' ones (Theorem~\ref{generator}). These results seem to be new even in the situations considered in \cite{Fel,Dod2,DM,Jor,Reu, Web, Woj}.

\smallskip

After these considerations, our  final  aim is to study how  $\SC$ of  a subgraph is related to $\SC$ of the whole graph. There, we obtain two results: We show that any graph is a subgraph of a  graph satisfying $\SC$. This completion can be achieved both by adding killing terms (Theorem~\ref{main1b}) and by adding edges (Theorem~\ref{main1}). We also show that stochastic incompleteness of a suitably modified subgraph implies stochastic incompleteness of the whole graph (Theorem~\ref{main2}). These results seem to be new  even in the contexts discussed earlier.

\smallskip

We have tried to make this paper as accessible and  self-contained as possible for both people with a background in Dirichlet forms and people with a background in geometry. For this reason some arguments are given,  which are certainly well known.

\smallskip

For further studies of certain spectral features of Laplacians in the framework developed below we refer the reader to \cite{HK,KL}, both of which were written after the present paper.

\smallskip

The paper is organized as follows. In Section~\ref{Framework} we present the
notation and our main results. A study of basic properties of Dirichlet forms
on graphs can be found in Section~\ref{Dirichlet}.
In Section~\ref{Essential}
we consider Dirichlet forms on graphs satisfying the condition $(A)$
mentioned above. For
these forms we calculate the generators of the $\ell^p$ semigroups for $p\in
[1,\infty)$ and  we show essential selfadjointness of the generators on $\ell^2$ (whenever the functions with compact support are in the domain of definition). In Section~\ref{Counterexamples} we give examples where essential selfadjointness fails as well as examples of non-regular Dirichlet forms on graphs.
A short discussion of the
heat equation in our framework is given in Section~\ref{heat}.
Section~\ref{Extended} deals
with extending the semigroup and resolvent in question to a somewhat larger
space of functions. In Section~\ref{Characterisation} we can then prove our
result characterizing  stochastic completeness for arbitrary
Dirichlet forms on graphs.
Section~\ref{Stochastically} contains a proof that any graph is a subgraph of
a stochastically complete graph and that any graph can be made stochastically complete by adding a killing term. Section~\ref{An}
contains an incompleteness criterion.

\section{Framework and results} \label{Framework}
Throughout $V$ will be a countable set.   Let $m$ be a measure on $V$ with full support
(i.e. $m$ is a map $m :V\longrightarrow (0,\infty)$). Then, $(V,m)$ is a measure space.  We will deal exclusively with real valued functions. Thus, $\ell^p (V,m)$, $1\leq p<\infty$, is defined by
$$ \{ u : V\longrightarrow \RR: \sum_{x\in V}  m(x)|u(x)|^p <\infty\}.$$
Obviously, $\ell^2 (V,m)$ is a Hilbert space with inner product given by
$$\langle u, v\rangle := \sum_{x\in V} m(x)u(x) v(x)\;\:\mbox{and norm}\;\:\aV{u}:=\langle u, u\rangle^\frac{1}{2}.$$
Moreover we denote by $\ell^\infty(V)$ the space of bounded functions on $V$. Note that this space does not depend on the choice of $m$. It is equipped with the supremum norm $\|\cdot\|_\infty$.

A symmetric non-negative form on
$(V,m)$ is given by a dense  subspace $D$ of $\ell^2 (V,m)$ called the domain of the
form and a map
$$ Q : D\times D \longrightarrow \RR  $$
with $Q(u,v) = Q(v,u)$ and $Q(u,u)\geq 0$ for all $u,v\in D$.
Such a map is already determined by its values on the diagonal.  For $u\in
\ell^2 (V,m)$ we then define $Q(u)$ by $Q(u):= Q(u,u)$ if $u\in D$ and
$Q(u):=\infty$ otherwise. If $\ell^2 (V,m)\longrightarrow [0,\infty]$,
$u\mapsto Q(u)$,  is lower semicontinuous, $Q$
is called closed.  If $Q$ has a closed extension, it is called closable and the
smallest closed extension is called the closure of $Q$.

A map $C:   \RR\longrightarrow \RR $ with $C(0) =0$ and
$|C(x) - C(y)|\leq |x - y|$ is  called a normal contraction. If $Q$ is both
closed and satisfies $Q(Cu) \leq Q(u)$ for all $u\in \ell^2 (V,m)$ and all normal contractions $C$,  it is
called a Dirichlet form on $(V,m)$ (see \cite{BH,Dav1,Fuk,MR} for background on   Dirichlet forms).

 Let $C_c (V)$ be the space of finitely supported functions on $V$. A Dirichlet $Q$ form on $(V,m)$ is called regular if $D(Q)\cap C_c (V)$ is  both dense in  $C_c (V)$ with respect to the supremum norm  and dense in  $D(Q)$ with respect to the form norm given by ${\|\cdot\|}_Q:=\sqrt{\|\cdot\|^2+Q(\cdot)}$. As discussed below, for such a regular form the set $C_c (V)$ is  necessarily contained in the form  domain. Thus, a  Dirichlet form  $Q$ is   regular if it is the closure of its
restriction to the subspace $C_c (V)$.

Regular Dirichlet forms on $(V,m)$ are
given by graphs on $V$, as we discuss next (see Section~\ref{Dirichlet} for details).
A symmetric weighted graph over $V$ or a symmetric  Markov
 chain on $V$ is  a pair  $(b,c)$ consisting of  a map $  b : V\times V\longrightarrow [0,\infty)$ with $b(x,x) =0$ for all $x\in V$ and a map $c : V\longrightarrow [0,\infty)$
satisfying the following two properties:

\begin{itemize}

\item[(b1)] $b(x,y)= b(y,x)$ for all $x,y\in V$.

\item[(b2)] $\sum_{y\in V} b(x,y) <\infty$ for all $x\in V$.
\end{itemize}
We can then  think of $(b,c)$ or rather  the triple  $(V,b,c)$ as a weighted  graph with vertex set $V$ in the following way: An $x\in V $ with
$c(x)\neq 0$ is then thought to be connected to the point $\infty$ by an edge with weight $c(x)$. Moreover, $x,y\in V$ with $b(x,y)>0$ are thought to be connected by an edge with weight $b(x,y)$.  The map $b$ is called the edge
weight. The map $c$ is called killing term.  Vertices $x,y\in V$ with $b(x,y)
>0$ are called neighbors.  More generally, $x,y\in V$ are called connected if
there exist $x_0,x_1,\ldots,x_n, x_{n+1}, \in V$ with $b(x_i, x_{i+1}) >0$,
$i=0,\ldots, n$ and $x_0 = x$, $x_{n+1} = y$. This allows us to define connected
components of $V$ in the obvious way.

To $(V,b,c)$ we associate the form $Q^{\comp}=Q^{\comp}_{b,c}$ on $C_c (V)$  with diagonal $Q^{\comp} : C_c (V) \longrightarrow [0,\infty]$ given by
$$
Q^{\comp}(u) = \frac{1}{2}
\sum_{x,y\in V} b(x,y) (u(x) - u(y))^2 + \sum_{x\in V} c(x) u(x)^2.$$
Obviously, $Q^{\comp}$ is a restriction of the form $Q^{\max}= Q^{\max}_{b,c,m}$ defined on $\ell^2 (V,m)$ with diagonal given by
$$  Q^{\max}  (u) = \frac{1}{2}
\sum_{x,y\in V} b(x,y) (u(x) - u(y))^2 + \sum_{x\in V} c(x) u(x)^2.$$
Here, the value $\infty$ is allowed. It is not hard to see that $Q^{\max}$ is closed  and hence  $Q^{\comp}$ is closable on   $\ell^2
(V,m)$ (see Section~\ref{Dirichlet}) and  the closure will be denoted by $Q=Q_{b,c,m}$ and
its domain by $D(Q)$ which is the closure of $C_c(V)$ with respect to $\|\cdot\|_Q$.  Then, there exists a unique selfadjoint operator $L = L_{b,c,m}$ on
$\ell^2 (V,m)$  such
that
$$ D(Q)  = \mbox{Domain of definition of $L^{1/2}$}$$
and
$$ Q(u) = \langle L^{1/2} u , L^{1/2} u\rangle$$
for $u\in D(Q)$ (see e.g.  Theorem 1.2.1 in \cite{Dav1}). As $Q$ is
non-negative so is $L$.  Moreover,  it is not hard to see that $Q^{\max} (Cu) \leq
  Q^{\max} (u)$ for all $u\in \ell^2 (V,m)$ (and in fact any function $u$)  and every normal contraction $C$.  Theorem
  $3.1.1$ of \cite{Fuk} then implies that $Q$ also satisfies $Q(Cu) \leq Q(u)$
 for all $u\in \ell^2 (V,m)$ and hence is a Dirichlet form. By construction it
 is regular. In fact, every regular  Dirichlet form on $(V,m)$ is of the form
 $Q=Q_{b,c,m}$ (see Theorem~\ref{characterizationDF} in Section~\ref{Dirichlet}).

\medskip

\noindent\textbf{Remark.} Our setting generalizes the setting of \cite{Dod2,DM,Jor,Web,Woj} to Dirichlet forms on
countable sets. In our notation, the situation of \cite{DM,Web,Woj}  can be described by the
assumptions $m\equiv 1$, $c\equiv 0$, and $b(x,y)\in \{0,1\}$ for all $x,y\in
V$ with $x\neq y$ and the setting of \cite{Dod2,Jor} can be described by  $m\equiv 1$, $c\equiv 0$ and $b(x,y) = 0$ for all but finitely many $y$ for each $x\in V$. In particular, unlike \cite{Dod2,DM,Jor,Web,Woj} we do not assume finiteness of the sets $\{y\in V: b(x,y)>0\}$ for all $x\in V$.

\medskip

Let now a measure $m$ on $V$ with full support and a weighted graph $(b,c)$
over $V$ be given. Let $Q$ be the associated form and $L$ its
generator.  Then, by standard theory \cite{Dav3,Fuk,MR}, the operators of the associated semigroup $e^{-tL}$, $t\geq 0$, and the associated  resolvent $\alpha (L +\alpha)^{-1}$, $\alpha >0$,
are  positivity preserving  and even markovian.

Positivity preserving means that they map non-negative functions to non-negative functions.  In
fact, if $(V,b,c)$ is connected they are even positivity improving, i.e.,  map non-negative nontrivial functions to positive functions (see Section~\ref{Dirichlet}). Markovian  means that they map non-negative
functions bounded by one to non-negative functions bounded by one.

This can be used to show that  semigroup and resolvent  extend to  all $\ell^p (V,m)$, $1\leq p\leq \infty$.  These  extensions  are consistent, i.e., two of them agree on their common domain and they are selfadjoint, i.e.,  the adjoint to the extension to $\ell^p (V,m)$ with $1\leq p < \infty$ is given by the extension to $\ell^q (V,m)$ for $1/p + 1/q = 1$, see  \cite{Dav1}.
The corresponding generators are denoted  by $L_p$. Thus, the extension  of $(L + \alpha)^{-1}$ to $\ell^p (V,m)$ is given by $(L_p + \alpha)^{-1}$.   In particular  we have $L = L_2$.

We can  describe the action of  the operator $L_p$ explicitly  (in  Section~\ref{Dirichlet}) as follows (see Theorem~\ref{obermenge}):
Define the formal Laplacian  $\widetilde{L} = \widetilde{L}_{b,c,m}$ on the vector space
\begin{equation}\label{ftilde}
\widetilde{F}:=\{ u : V\longrightarrow \RR : \sum_y |b(x,y) u(y)|<\infty\;\mbox{for all
  $x\in V$ } \}
  \end{equation}
 by
$$\widetilde{L} u (x) :=\frac{1}{m(x)} \sum_y b(x,y) (u(x) - u(y)) +
\frac{c(x)}{m(x)}  u(x),$$
where, for each $x\in V$,  the sum exists  by assumption on $u$. Then,
 $L_p$  is a restriction of $\widetilde{L}$ for any $p\in[1,\infty]$.

After having discussed the fact that these are different semigroups on different $\ell^p$ spaces, we will now follow the custom and  write $e^{-t L}$ for all of them.

The preceding  considerations show that
$$ 0 \leq e^{-t L} 1 (x)\leq 1$$
for all $t\geq 0$ and $x\in V$. The question, whether the second inequality is actually an equality  has received quite some attention.  In the case of vanishing killing term, this is discussed under the name of stochastic completeness or conservativeness.
In fact, for $c\equiv 0$ and $b(x,y)\in \{0,1\}$ for all $x,y\in V$, there is a  characterization of stochastic completeness  of Wojciechowski \cite{Woj} (see \cite{Dod2,DM,Web} for related results as well).
This
characterization is an analogue to corresponding results for Markov processes \cite{Fel,Reu}, results on manifolds of Grigor'yan \cite{Gri} and results of  Sturm for general strongly
local Dirichlet forms \cite{Stu}.

Our first main result
concerns a version of this result for  arbitrary regular Dirichlet forms on weighted graphs (see Section~\ref{Characterisation} for details and proofs concerning the subsequent discussion):
In this case, we have to replace $e^{-t L} 1 $ by the function
$$ M_t (x) := e^{-tL} 1 (x) + \int_0^t ({e^{-s L} \frac{c }{m}})(x) ds, \quad x\in V. $$
This  is well defined,  satisfies $0\leq M \leq 1$ and for each $x\in V$, the function $t\mapsto M_t (x)$ is continuous and even differentiable.  Note that for
$c\equiv 0$, we obtain $M = e^{-t L} 1$ whereas for $c\neq 0$ we obtain
$M_t >e^{- t L} 1$ on any connected component of $V$ on which $c$ does
not vanish identically (as the semigroup is positivity improving). The term $e^{-t L}1$ can be interpreted as the amount of heat contained in the graph at time $t$ and the integral can be interpreted as the amount of heat killed within the graph up to the time $t$. Thus, $1- M_t$ is the amount of heat transported to the boundary of the graph by the time $t$ and
$M_t$  can be interpreted as the
amount of heat, which has not been transported to the boundary of the graph at
time $t$. Our question then becomes whether the quantity
$$ 1 - M_t$$
vanishes identically or not.  Our result then reads.

\begin{theorem}\label{main0} (Characterization of heat transfer to the
  boundary) Let $(V,b,c)$ be a weighted graph and $m$ a measure on $V$ of full
  support. Then, for any $\alpha >0$, the
  function
$$ w := \int_0^\infty \alpha e^{-t \alpha } ( 1 - M_t) dt$$
satisfies $0\leq w \leq 1$, solves  $(\widetilde{L} + \alpha) w = 0$, and
is the largest non-negative $l\leq 1$ with $(\widetilde{L} + \alpha) l
\leq  0$.  In particular, the following assertions are equivalent:

\begin{itemize}
 \item[(i)] For any $\alpha>0$ there exists  a nontrivial, non-negative bounded $l$ with  $(\widetilde{L} + \alpha) l \leq 0$.
\item[(ii)] For any $\alpha >0$ there exists a nontrivial bounded $l$ with $(\widetilde{L} + \alpha) l = 0$.
\item[(iii)] For any $\alpha >0$ there exists a nontrivial, non-negative bounded $l$ with  $(\widetilde{L} + \alpha) l = 0$.
\item[(iv)] The function  $w$ is nontrivial
\item[(v)] $M_t (x) <1$ for some $x\in V$ and some  $t>0$.
\item[(vi)] There exists a nontrivial, non-negative  bounded $N : V\times
  [0,\infty)\longrightarrow [0,\infty)$ satisfying
  $\widetilde{L} N + \frac{d}{d t}  N =0$ and  $N_0 \equiv 0$.
\end{itemize}
\end{theorem}

\noindent\textbf{Remark.} (a) Conditions (ii) and (iii) deal with eigenvalues of $\widetilde{L}$ considered as  an operator on $\ell^\infty (V)$. In particular, (ii) must fail (for sufficiently large $\alpha$) whenever $\widetilde{L}$ gives rise to a bounded operator on $\ell^\infty (V)$. Thus, any bounded operator $\widetilde{L}$ yields a stochastically complete graph. In this way we recover the corresponding results of \cite{Dod2,DM}.

(b) The case $c\equiv 0$ $m\equiv 1$, $b(x,y)\in \{0,1\}$ recovers the corresponding result of \cite{Woj}. In fact, in  the case $c\equiv 0$, $m\equiv 1$ and general (not even symmetric)   $b$  the equivalence of (i) (or (ii)) and (v) is  already  discussed in \cite{Fel,Reu}. These works mainly  aim at studying uniqueness of the Markov process, i.e., a (somewhat weaker) version of (vi).  They  characterize this uniqueness by validity of (i) for $m\equiv 1$ and arbitrary $c$.  In this sense it seems fair to say that for $c\equiv 0$ the equivalence of (i), (v) and (vi) is well known and for general $c$ the equivalence of (i) and (vi) is well known.  Besides our new proof (inspired by \cite{Woj,Gri}), our main contribution  here is the definition of $M$ allowing for an  extension of (v)  to situations with killing terms.

(c) The minimum principle discussed below, Theorem \ref{minimum-principle},  will show that for $\alpha >0$  any  nontrivial, non-negative  solution $u$ of $(\widetilde{L} + \alpha) u = 0$ satisfies $u>0$ if the underlying graph is connected.

(d) Let $L$ be the operator associated to a weighted graph $(V,b,c)$ and $L_0$ the operator associated to $(V,b,0)$, both with respect to the same  measure $m:V\to(0,\infty)$. The equivalence of (i) and (v) in the  theorem above obviously implies $M_t =1$ whenever $e^{-tL_0}1=1$, since $\ow L l\geq \ow L_0l$ for every non-negative $l\in \ow F$.
\medskip

The previous  theorem suggests the following definition for stochastic completeness at infinity and stochastic incompleteness at infinity for general  Dirichlet forms on weighted graphs.

\begin{definition}
The weighted graph $(V,b,c)$ with the measure $m$ of full support is said to satisfy $\SI$ if it satisfies one (and thus all) of the equivalent assertions  of Theorem~\ref{main0}. Otherwise $(V,b,c)$ is said to satisfy $\SC$.
\end{definition}

\noindent\textbf{Remark.} Note that validity of $\SI$ depends on both $(V,b,c)$ and
$m$. In fact, for given $(V,b,c)$ it is always possible to choose $m$ in such
a way that $\widetilde{L}$ becomes a bounded operator on $\ell^\infty
(V)$. Then, $\SC$ holds (by the previous remark).

\smallskip

The following two results show how  graphs can be made to satisfy $\SC$ by addition of killing terms or edges.
They  seem to be new even in the setting considered in \cite{Dod2,DM,Web,Woj}.

\begin{theorem}\label{main1b} Let $m$ be a measure on $V$ with full support.  For any weighted graph $(V,b,c)$ there is $c':V\to[0,\infty)$ such that $(V,b,c+c')$ satisfies $\SC$.
\end{theorem}

\noindent\textbf{Remark.} Of course, addition of killing terms yields to loss of mass from the graph reflected in the inequality $e^{-t L} 1< 1$.  As our concept of $\SC$ only considers  mass transported to the geometric boundary of the graph, we can have and even enforce $\SC$ by adding killing terms. More precisely, the
 theorem can  be understood in the following way: Adding a killing term kills heat within the graph on any vertex where the killing term does not vanish. If we eliminate enough heat by the killing terms, we can achieve that no more heat is transferred to the geometric boundary of the graph.

\smallskip

A subgraph $(W,b_W,c_W)$ of a weighted graph  $(V,b,c)$ is given by a subset $W$ of $V$ and the restriction $b_W$ of $b$ to $W\times W$ and the restriction $c_W$ of $c$ to $W$.
The weighted graph $(V,b,c)$ is then called a supergraph to $(W,b_W,c_W)$.  Given a
measure $m$ on $V$ we denote its restriction to $W$ by $m_W$. The subgraph
$(W,b_W,c_W)$ then  gives rise to a form on $\ell^2
(W, m_W)$  with  associated operator $L_{b_W,c_W,m_W}$.

\begin{theorem}\label{main1} Any weighted graph is the subgraph of a weighted graph satisfying $\SC$. This supergraph can be chosen to have vanishing killing term if the original graph has vanishing killing term.
\end{theorem}

\begin{remark}{\rm Note that (in the common definitions)  the volume growth of balls in a graph dominates the volume growth of
balls in any of its subgraph. Thus, the  theorem has the  consequence that  failure of $\SC$ can not be inferred from lower bounds  on the growth of volumes of ball.}
\end{remark}

While subgraphs do not force incompleteness according to  Theorem~\ref{main1},
suitably adjusted subgraphs do force incompleteness of the whole graph. In
order to be more precise, we need some more notation.

Let $(V,b,c)$ be a weighted graph  with measure $m$ of full support and  $W$ a subset of $V$. Let $m_W$ be the restriction of $m$ to $W$.  Let   $i_W : \ell^2
(W,m_W)\longrightarrow \ell^2 (V,m)$ be  the canonical embedding, i.e.,  $i_W (u)$ is the extension of $u$ to $V$ by setting $i_W (u)$ identically zero outside of $W$. Let $p_W : \ell^2 (V,m)
\longrightarrow \ell^2 (W,m_W)$ be the canonical projection, i.e., the adjoint of $i_W$.  Then, $W$ gives rise to the
form $Q_{W}^{\comp, (D)}$  defined  on $C_c (W)$  by
$$ Q_{W}^{\comp, (D)}  (u) := Q ( i_W u) =  Q_{b_W,c_W}^{\comp} (u) +\sum_{x\in W} d_W (x)
u(x)^2 .$$
Here, $d_W (x) := \sum_{y\in V \setminus W} b(x,y)$ describes the
   edge deficiency  of  vertices in $W$ compared to the same vertex in $V$.
Thus, $Q_{W}^{\comp, (D)}$ is in fact the form  $Q^{\comp}$ of  the weighted graph $(W,b_W^{(D)},c_W^{(D)})$ with
$$b_W^{(D)} = b_W \:\quad \mbox{and}\quad \: c_W^{(D)} = c_W + d_W.$$
In particular, by the theory developed above, its closure in $\ell^2 (W,m_W)$,  denoted by  $Q_{W}^{(D)}$,   is a Dirichlet form.
The associated selfadjoint operator will be denoted by $L_{W}^{(D)}$.
%It is given by
%$$
%L_{W}^{(D)} = p_W L i_W.$$
This operator is sometimes thought of as  a restriction of the original operator to $W$ with Dirichlet
boundary condition. For this reason we include the superscript  $D$
in the notation. Another interpretation (suggested by the above
expression for the form) is to think about the graph which arises from the subgraph $W$ by adding one way edges to a vertex at infinity according to the mentioned edge deficiency.

Again, it  is not hard to
express the action of  $L_{W}^{(D)}$ explicitly. In fact, the above considerations applied to the graph  $(W,b_W^{(D)},c_W^{(D)},m_W)$ show that
$$   L_{W}^{(D)} u =  \widetilde{L}_W^{(D)} u $$
for any $u\in D(L_{W}^{(D)})$. Here,
the formal Dirichlet Laplacian  $  \widetilde{L}_W^{(D)}$  on $W$ is defined on  $p_W \widetilde{F} = i_w^{-1} (\widetilde{F})$ and given by
\begin{align*}
{\widetilde{L}_W^{(D)} u (x)=\frac{1}{m(x)}\left( \sum_{y\in W} b^{(D)}_{W}(x,y)(u(x)-u(y))+c^{(D)}_{W}(x) u(x)\right)= \widetilde{L} i_W u (x)}
\end{align*}
for $x\in W$.  These considerations give that for a function $u$ on $W$ (which is extended by $0$ to $V$) the equality
\begin{equation}
\widetilde{L}_W^{(D)} u (x)  = \widetilde{L} u (x)
\end{equation}
holds for any $x\in W$. This will be used repeatedly in the sequel.  Note also  that for $W=V$ we recover the operator on the whole graph, i.e., $\widetilde{L}_V^{(D)} = \widetilde{L}$ and $
L_{V}^{(D)} = L$.

\smallskip

The following result seems to be new even in the setting considered in \cite{Dod2,DM,Web,Woj}.

\begin{theorem}\label{main2}
  Let $(b,c)$ be a weighted graph over $V$ and $m$ a measure on $V$ of full
  support. Then $\SI$ holds, whenever there exists $W\subseteq V$ such
  that the weighted graph $(b_W^{(D)},c_W^{(D)})$ over the measure space
  $(W,m_W)$ satisfies $\SI$.
\end{theorem}

As an example of a situation in which the theorem may be applied we note the following consequence.

\begin{coro}
  \label{beispiel} Let $(b,0)$ be a weighted graph over $V$ with vanishing killing term and $m$ a measure on $V$ of full   support. Let $W$ be a subset of $V$ such that $(b_W^{(D)}, 0)$ over the measure space $(W,m_W)$ satisfies $\SI$ and there exists $C>0$ with $\sum_{y\in V\setminus W} b(x,y)/m(x)\leq C$ for any $x\in W$. Then  $(b,0)$ over $(V,m)$ satisfies $\SI$.
\end{coro}

So far, we have not discussed  the precise domains of definition for our
operators.  In fact, the actual domains have been quite irrelevant for our considerations.

To determine the domains we need a geometric condition saying that any infinite path has infinite measure. More precisely, we define condition $(A)$ as follows:
\begin{itemize}
\item[$(A)$] The equality $\sum_{n\in \NN} m (x_n) = \infty$ holds for any sequence $(x_n)$ of elements of $V$ such that $b(x_n, x_{n+1})>0$ for all $n\in \NN$ .
\end{itemize}
Of course, an equivalent requirement would be that the equality $m (\{x_n : n\in \NN\})=\infty$ holds
for any sequence $(x_n)$ of pairwise different elements of  $V$ such that $b(x_n,x_{n+1})>0$ for all $n\in \NN$.

Note that $(A)$ is a condition on $(V,m)$ and $(b,c)$
together. If
$$\inf_{x\in V}  m(x) >0$$
holds, then  $(A)$ is satisfied for all weighted graphs $(b,c)$ over $V$.

\medskip

Our  result reads as follows. We are not aware of an earlier  result of this form in this context.

\begin{theorem}\label{generator} Let $(V,b,c)$ be a weighted graph and $m$ a
 measure on $V$ of  full support such that $(A)$ holds. Then, for any $p\in [1,\infty)$ the operator $L_p$ is the
  restriction of $\widetilde{L}$ to
  $$   D(L_p)=\{ u\in \ell^p (V,m) : \widetilde{L} u \in \ell^p
  (V,m)\}. $$
\end{theorem}
\noindent\textbf{Remark.} The theory of Jacobi matrices already provides examples
showing that without $(A)$ the statement becomes false for $p=2$. This is discussed in Section~\ref{Counterexamples}.

The condition $(A)$ does not imply that $\widetilde{L} f$ belongs to
$\ell^2 (V,m)$ for all $f\in C_c (V)$.
However,  if this is the case, then $(A)$ does imply essential
selfadjointness. In this case, $Q$ is the ``maximal'' form associated to the weighted
graph $(b,c)$. More precisely, the following holds.

\begin{theorem}\label{essential}
Let $V$ be a set,  $m$ a measure on $V$ with full support,  $(b,c)$ a weighted graph
over $V$ and $Q$ the associated regular  Dirichlet form. Assume
$\widetilde{L} C_c (V) \subseteq  \ell^2 (V,m)$. Then,   $D(L)$ contains $C_c (V)$. If furthermore $(A)$ holds, then the restriction of $L$ to $C_c (V)$  is essentially selfadjoint and the domain of $L$  is given by
$$D (L) = \{u\in \ell^2 (V,m) : \widetilde{L} u\in \ell^2 (V,m)\}$$
and the associated form  $Q$ satisfies  $Q = Q^{\max}$, i.e.,
$$ Q(u) = \frac{1}{2} \sum_{x,y\in V} b(x,y) (u(x) - u(y))^2 + \sum_{x\in V} c(x)u(x)^2$$
for all $u\in \ell^2 (V,m)$.
\end{theorem}

\noindent\textbf{Remark.} (a)  If $\inf m(x) >0$ then both  $(A)$ and $\widetilde{L} C_c (V)
\subseteq  \ell^2 (V,m)$ hold  for any weighted graph $(b,c)$ over $V$. In this case, we recover the corresponding results of  \cite{Jor,Web, Woj} on  essential selfadjointness, as these works assume $m\equiv 1$. (They also have additional restrictions on $b$ but this is not relevant here).

(b)   The statement on the form being the maximal one seems to be new even in the context of \cite{Jor,Web, Woj}

(c) Essential selfadjointness fails in general if $(A)$ does not hold as can be seen by examples (see Section~\ref{Counterexamples} and previous remark).
\medskip

\section{Dirichlet forms on graphs - basic facts} \label{Dirichlet}
In this section we consider a countable set $V$ together  with a
measure $m$ of full support.

\begin{lemma} Let $Q$ be a regular  Dirichlet form on $(V,m)$. Then, $C_c (V)$ is contained in $D(Q)$.
\end{lemma}
\begin{proof}  Let $x\in V$ be arbitrary. Choose $\ph \in C_c (V)$ with $\ph (x) =2 $ and $\ph (y) =0$ for all $y\neq x$.   As $C_c (V)\cap D(Q)$ is dense in $C_c (V)$ with respect to the supremum norm,  there exists  $\psi\in D(Q)$ with  $\psi (x)>1$ and $|\psi (y)|<1$ for all $y\neq x$.  As $Q$ is a Dirichlet form, $D(Q)$ is invariant under taking modulus and we can assume that  $\psi$  is non-negative. As $Q$ is a Dirichlet form, also $\widetilde{\psi}:=\psi \wedge 1$ belongs to $D(Q)$. (Here, $\wedge$ denotes the minimum.) As $D(Q)$ is a vector space it contains  $\psi - \widetilde{\psi}$ and this is a (nonzero) multiple of $\ph$ by construction. As $x\in V$ was arbitrary, the statement follows.
\end{proof}

\begin{lemma}\label{enthalten}  Let $Q$ be a regular Dirichlet form on $(V,m)$. Then, there exists a
  weighted graph $(b,c)$ over $V$ such that the restriction of $Q$ to $C_c
  (V)$ equals $Q^{\comp}_{b,c}$.
\end{lemma}
\begin{proof} By the previous lemma, $C_c (V)$ is contained in $D(Q)$.
Then, for any finite $K\subseteq V$,  the restriction $Q_K$ of $Q$ to $C_c (K)$ is a Dirichlet form  as well.  By standard results (see
e.g. Th\'eor\`eme 1 in \cite{BD}), there exists then $b_K$, $c_K$ with $Q_K =
  Q^{\comp}_{b_K,c_K}$. For $K\subseteq K'$ and $x,y\in K$ it is not hard to see that
  $b_K(x,y) = b_{K'}(x,y)$ and $c_K(x) \geq c_{K'} (x)$.
Thus, a simple limiting procedure gives the result.
\end{proof}

\begin{lemma}\label{closable} Let $m$ be a measure on $V$ of full support.
 Let $(b,c)$ be a weighted graph over
  $V$. Then, $Q^{\max}_{b,c,m}$ is closed and  $Q^{\comp}_{b,c}
$ is closable and its  closure $Q_{b,c,m}$  is a restriction of  $Q^{\max}_{b,c,m}$.
\end{lemma}
\begin{proof} It suffices to show that $Q^{\max}_{b,c,m}$ is closed. Thus,
  it suffices to show lower semicontinuity of $u\mapsto Q^{\max}_{b,c,m} (u,u)$. This
  follows easily from follows easily from Fatou's lemma.
\end{proof}

\begin{theorem}\label{characterizationDF} The regular Dirichlet forms on $(V,m)$ are exactly given by
  the forms $Q_{b,c,m}$ with weighted graphs $(b,c)$ over $V$.
\end{theorem}
\begin{proof} By the previous lemma and the discussion in Section~\ref{Framework}, any $Q_{b,c,m}$ is a regular Dirichlet form. The converse
  follows from the previous lemmas.
\end{proof}

The study of regular Dirichlet forms on $(V,m)$ is based on first understanding their restrictions to finite sets. This is done next.

\begin{prop} \label{finite-restriction}  Let $(V,m)$ be given and $(b,c)$ a weighted graph over $V$. Let $K\subset
  V$ be finite. Then, $L_K^{(D)}$ is a bounded operator with
$$ L_K^{(D)} f (x) = \frac{1}{m(x)}\ab{ \sum_{y\in K} b(x,y) (f(x) - f(y)) + \ab{\sum_{y\in V\setminus K}b(x,y)+ c(x)}f(x)}. $$
In particular,   $\widetilde{L} i_K f (x) =  L_K^{(D)} f (x)$ for all $x\in K$, where $i_K :\ell^2 (K,m_K)\longrightarrow \ell^2 (V,m)$ is the canonical embedding by extension by zero.
\end{prop}
\begin{proof} Every linear operator on the finite dimensional $\ell^2 (K,m_K)$
  is bounded. Thus, we can directly read off the operator $L_K^{(D)}$ from the
  form $Q_K^{(D)}$ given by $Q_K^{(D)} (u) := Q (i_K u)$. This gives the first claim. The last statement follows easily.
\end{proof}

We now discuss  two results on solutions of the associated difference
equation. These results will be rather useful for our further
considerations. We start with a version of a minimum principle.

\begin{theorem} (Minimum principle) \label{minimum-principle}  Let  $(V,b,c)$ be a weighted graph and $m$ a measure on $V$ of full support.  Let $U\subseteq V$
  be given. Assume that the function  $u$ on  $V$ satisfies
\begin{itemize}
 \item $(\widetilde{L} + \alpha)  u \geq 0$ on  $U$ for some $\alpha>0$,
\item  The negative part $u^-_U:=u\vert_U\wedge0$ of the restriction of $u$ to $U$ attains its minimum,
\item $u \geq 0$ on  $V\setminus U$.
\end{itemize}
Then,   $u\equiv  0$ or $u>0$ on each connected component of $U$. In
particular,  $u \geq 0$.
\end{theorem}
\begin{proof} Without loss of generality we can assume $U$ is   connected. If $u>0$ there is nothing left to show.  It remains to consider the case that there exists $x \in U$ with  $u(x)\leq 0$. As the negative part of $u$ on $U$ attains its minimum, there exists then $x_m\in  U$ with $u(x_m) \leq 0$ and  $u(x_m)\leq u(y)$ for all $y\in U$.  As $u(y)\geq 0$ for $y\in U^c$, we obtain   $u(x_m) - u(y) \leq0$ for all  $y\in V$. By the supersolution assumption we find
$$ 0 \leq \sum b(x_m,y) (u(x_m) - u(y)) + c(x_m) u(x_m) + \alpha  m(x_m) u(x_m)  \leq 0.$$
As $b$ and $c$ are non-negative and $\alpha>0$, we find $0 = u(x_m)$ and
$u(y) = u(x_m) =0$ for all $y$ with $b(y,x_m)>0$. As   $U$  is connected, iteration of this argument shows $u\equiv 0 $ on $U$.
 \end{proof}

The following lemma will be a key tool in our investigations. Note that its proof is rather simple due to the discreteness of the underlying space.

\begin{lemma} (Monotone convergence of solutions) Let $\alpha \in \RR$,
  $f: V\lra \RR$ and $u: V\longrightarrow \R$ be given. Let
  $(u_n)$
  be a sequence of  non-negative functions on $V$ belonging  to the set $\widetilde{F}$ given in \eqref{ftilde} on which $\widetilde{L}$ is defined. Assume $u_n \leq u_{n+1}$ for all $n\in \N$,  and $u_n (x) \to u (x)$
  and $(\widetilde{L}+\alpha) u_n (x) \to f(x)$ for all $x\in V$. Then, $u$ belongs to $\widetilde{F}$ as well and  the equation
  $(\widetilde{L}+\alpha) u = f$ holds.
\end{lemma}
\begin{proof} Without loss of generality we assume $m\equiv 1$.
 By assumption
$$ (\widetilde{L} + \alpha)  u_n (x) = \sum_{y\in V} b(x,y) (u_n (x) -u_n (y)) +
(c(x)+\alpha)  u_n(x)$$
converges to $f(x)$ for any $x\in V$. As  $\sum_{y\in V} b(x,y) u_n
(x)$ converges increasingly to $u(x) \sum_{y\in V} b(x,y) <\infty$, the assumptions on $u_n$  show that $\sum_{y\in V} b(x,y)  u_n (y)$ must converge as well and  in fact must converge to $\sum_{y\in V} b(x,y) u(y)$ by monotone converges theorem. From this we easily
obtain the statement.
\end{proof}

We next discuss some fundamental properties of regular Dirichlet forms. These
properties do not depend on the graph setting.  They are true for general
Dirichlet forms and can, for example,  be found in the works
\cite{Sto,SV}. For the convenience of the reader we  include  short proofs
based on the previous  minimum principle.

\begin{prop} (Domain monotonicity)
Let $(V,b,c)$  be a weighted graph and $m$ a measure on $V$ of full support.  Let $K_1, K_2  \subseteq V$  finite  with  $K_1 \subseteq K_2$ be given. Then, for any $x\in K_1$
$$(L^{(D)}_{K1}+\alpha)^{-1}f(x)\leq(L^{(D)}_{K_2}+\alpha)^{-1}f(x)$$
for all $f\in \ell^2 (V,m)$ with $f\geq 0$ and $\supp f \subseteq K_1$.
\end{prop}
\begin{proof}Consider $f\in \ell^2 (V,m)$ with $f\geq 0$ and $\supp f \subseteq  K_1$ and define $u_i :=(L^{(D)}_{K_i} + \alpha)^{-1} f$, $i=1,2$. Extending   $u_i$ by zero we can assume that $u_i$ are defined on the whole of  $V$. Then,
$$ (\widetilde{L} + \alpha) u_i = f\:\;\mbox{on $K_i$} $$
for $i=1,2$. Therefore,  $w:= u_2 - u_1$ satisfies
\begin{itemize}
\item $w = u_2 \geq 0$ on $K_1^c$.
\item The negative part of $w$ attains its minimum on $K_1$ (as $K_1$ is finite).
\item $(\widetilde{L} + \alpha) w = f - f = 0$ on $K_1$.
\end{itemize}
The minimum principle yields $w\geq 0$ on  $V$.
\end{proof}

Regularity is crucial for the proof of the following result.

\begin{prop}\label{convergence}(Convergence of resolvents/semigroups)
Let $(V,b,c)$  be a weighted graph, $m$ a measure on $V$ with full support  and $Q$ the associated regular Dirichlet  form. Let   $(K_n)$ be an increasing  sequence of finite  subsets  of $V$ with $V =\bigcup K_n$.   Then,  $  (L_{K_n}^{(D)} + \alpha)^{-1}  f \to (L+\alpha)^{-1} f, \;\: n\to \infty$ for any $f\in \ell^2 (K_1,m_{K_1})$. (Here, $  (L_{K_n}^{(D)} + \alpha)^{-1}  f $ is
extended by zero to  all of $V$). The corresponding statement also holds for the semigroups.
\end{prop}
\begin{proof} By general principles (see e.g. \cite[Satz 9.20b]{Wei}) it suffices to consider the resolvents.   After decomposing $f$ in positive and negative part, we can
  restrict attention to $f\geq 0$. Define  $u_n := (L_{K_n}^{(D)} +
  \alpha)^{-1}   f$. Then, $u_n\geq 0$.
Now, by standard  characterization of resolvents (see e.g. Section~$1.4$ in \cite{Fuk}), $u_n$ is the unique minimizer of
$$Q_{K_n} (u) + \alpha \| u - \frac{1}{\alpha} f\|^2.$$
By domain monotonicity, the sequence  $(u_n (x))$  is   monotonously increasing for any $x\in V$. Moreover, by standard results on   Dirichlet forms, (see e.g. \cite[Theorem 1.4.1]{Fuk}) we have $u_n  \leq  \frac{1}{\alpha} \|f\|_\infty$ and by the spectral theorem $\|u_n\|\leq  \frac{1}{\alpha} \|f\|$.  Thus,  the sequence $u_n$ converges pointwise and in $\ell^2 (V,m)$ towards a function $u\in \ell^2 (V,m)$. Let now $w\in C_c (V)$ be arbitrary. Assume without loss of generality that the support of $w$ is contained in $K_1$. Then, $Q(w) = Q_{K_n} (w)$ for all $n\in \N$.
Closedness of $Q$, convergence of the $(u_n)$ and the minimizing property of each $u_n$ then  give
\begin{eqnarray*}
Q(u) + \alpha \| u - \frac{1}{\alpha} f\|^2 &\leq & \liminf_{n\to \infty} Q
(u_n) + \alpha \| u - \frac{1}{\alpha} f\|^2\\
&=& \liminf_{n\to \infty} \left(  Q(u_n) + \alpha \| u_n - \frac{1}{\alpha} f\|^2\right)\\
&=& \liminf_{n\to \infty}  \left( Q_{K_n} (u_n)  + \alpha \| u_n - \frac{1}{\alpha} f\|^2\right)\\
&\leq &\liminf_{n\to\infty} \left( Q_{K_n} (w) + \alpha \| w - \frac{1}{\alpha} f\|^2\right)\\
&=& Q (w) + \alpha \| w - \frac{1}{\alpha} f\|^2.
\end{eqnarray*}
As $w\in C_c (V)$ is arbitrary and $Q$ is regular (!), this implies
$$  Q(u) + \alpha \| u - \frac{1}{\alpha} f\|^2 \leq Q(v) + \alpha \| v -
\frac{1}{\alpha} f\|^2$$
for any $v\in D(Q)$. Thus, $u$ is a minimizer of
$$ Q(u) + \alpha \| u - \frac{1}{\alpha} f\|^2.$$
By
  characterization of resolvents again,  $u$ must then be equal to $(L +
  \alpha)^{-1} f$.
\end{proof}

We can use the previous result to connect the operator $L$ to the formal operator
$\widetilde{L}$.  To do so we need one  further result.

\begin{lemma}\label{lzwei}  Let $(V,m)$ be given and $(b,c)$ a weighted graph over $V$.
Let $p\in [1,\infty]$ be given.  For any $g\in \ell^p (V,m)$, the function $u:=(L_p +
  \alpha)^{-1} g$ belongs to the set $\widetilde{F}$ given in \eqref{ftilde} on which $\widetilde{L}$ is defined and solves $(\widetilde{L} + \alpha )u = g$.
\end{lemma}
\begin{proof} We first consider the case $p=2$. If suffices to consider the case $f\geq 0$. Choose an increasing
  sequence $(K_n)$ of finite subsets of $ V$ with $\bigcup K_n = V$ and let $g_n$
  be the restriction of $g$ to $K_n$. Then, $(g_n)$ converges monotonously increasing to $g$ in
  $\ell^2 (V,m)$ and consequently $(L + \alpha)^{-1} g_n$
  converges monotonously increasing to $u$. Thus, by monotone convergence of
  solutions,  we can assume without loss of generality that
   $g$ has  compact support contained in $K_1$. By convergence of resolvents,
  $u_n:= (L_{K_n}^{(D)} + \alpha)^{-1}  g$ then converges increasingly to  $ u:=(L +\alpha)^{-1} g.$
Moreover, by Proposition \ref{finite-restriction},
$ u_n$ satisfies $(\widetilde{L} + \alpha) u_n = g$ on $K_n$. Thus, the statement follows
by monotone convergence of solutions.

We now turn to general $p\in [1,\infty]$.  Again, it suffices to consider the case $g\geq 0$. Choose an increasing
  sequence $(K_n)$ of finite subsets of $ V$ with $\bigcup K_n = V$ and let $g_n$
  be the restriction of $g$ to $K_n$. Then, $ u_n := (L_p + \alpha)^{-1} g_n$
  converges to $u$. Moreover, as $g_n$  belongs to $\ell^2 (V,m)$ consistency of the
  resolvents gives $ u_n = (L + \alpha)^{-1} g_n$. Now, on the $\ell^2
  (V,m)$ level we can apply the considerations for $p=2$  to obtain  $$(\widetilde{L} + \alpha) u_n = (\widetilde{L} + \alpha)(L + \alpha)^{-1} g_n = g_n.$$
Taking monotone limits now yields the statement.
\end{proof}

After these preparations, we  can now give the desired information on the generators.

%As already discussed in the introduction, the semigroup associated to a graph on $\ell^{2}$ can be extended to a semigroup $T_{t}^{(p)}$, $t\geq 0$, on $\ell^{p}$ for $1\leq  p<\infty$. The generator $L_{p}$ of $T_{t}^{(p)}$ has domain given by
%\begin{align*}
%    D(L_{p})=\{u\in\ell^{p}: \lim_{t\to0} \tfrac{1}{t}(T_{t}^{(p)}-I)u\mbox{ exists in }\ell^{p}\}
%\end{align*}
%(with $I$ being the identity operator) and acts by
%\begin{align*}
%    L_{p}u=-\lim_{t\to0} \tfrac{1}{t}(T_{t}^{(p)}-I)u
%\end{align*}

\begin{theorem}\label{obermenge} Let $(V,b,c)$ be a weighted graph and $m$ a  measure on $V$ of  full support. Let  $p\in [1,\infty]$ be given.  Then,  $L_p f = \widetilde{L} f$ for any $f\in D(L_p)$.
\end{theorem}
\begin{proof}
 Let $f\in D(L_p)$ be given. Then, $g:=(L_p + \alpha)f$
  exists and belongs to $\ell^p (V,m)$. By the previous lemma,  $f= (L_p +\alpha)^{-1} g$ solves
$$ (\widetilde{L} + \alpha) f = g = (L_p + \alpha)f$$
and we infer the statement.
\end{proof}

We also note the following by product of our investigation (see \cite{Web, Woj,Dav3} for this result for locally finite graphs).

\begin{coro}\label{positivityimproving} (Positivity improving)  Let $(V,b,c)$ be a  connected weighted graph and $L$ the associated operator. Then, both the semigroup $e^{-t L}$, $t\geq 0$, and the resolvent $(L + \alpha)^{-1}$, $\alpha>0$, are positivity improving (i.e., they map non-negative nontrivial $\ell^2$-functions to strictly positive functions).
\end{coro}
\begin{proof} By general principles it suffices to consider the resolvent.
Let $f\in\ell^2(V,m)$ with  $f\geq 0$ be given and consider $u:= (L + \alpha)^{-1} f$.  Then
  $u\geq 0$ as the resolvent of a Dirichlet form is positivity preserving.
 If
  $u$ is not strictly positive, there exists an $x$ with $u(x) =0$. As $u$ is
  non-negative, $u$ attains its minimum in $x$. By Lemma \ref{lzwei}, $u$ satisfies $(\widetilde{L}+ \alpha) u= f\geq 0$. We can therefore apply the
 minimum principle
  (with $U=V$) to obtain  that $u\equiv 0$. This implies $f\equiv 0$.
\end{proof}

\section{Generators of the semigroups on $\ell^p$ and essential
  selfadjointness on $\ell^2$} \label{Essential}
In this section we will consider a symmetric weighted graph $(V,b,c)$ and a measure $m$ on $V$ of full support. We will  be concerned with explicitly describing the generators of the semigroups on $\ell^p$ and  studying essential selfadjointness of the
generator on $\ell^2$. Both issues will be tackled by proving uniqueness of
solutions on the corresponding $\ell^p$ spaces.  The results of this section
are not needed to deal with stochastic completeness.

Recall the  geometric assumption introduced in the first section:
\begin{itemize}
\item[$(A)$] The equality $\sum_{n\in \NN} m (x_n) = \infty$ holds for any sequence $(x_n)$ of elements of $V$ such that $b(x_n, x_{n+1})>0$ for all $n\in \NN$ .
    \end{itemize}

The relevance of $(A)$ comes from the following variant of the minimum principle:

\begin{prop} Assume $(A)$. Let $\alpha>0$, $p\in [1,\infty)$
  and $u\in \ell^p (V,m)$ with $(\widetilde{L} + \alpha) u \geq 0$ be
  given. Then, $u\geq 0$.
\end{prop}
\begin{proof} Assume the contrary. Then, there exists an $x_0\in V$ with
  $u(x_0)<0$. By
$$ 0\leq (\widetilde{L} + \alpha) u (x_0) = \frac{1}{m (x_0)}\sum_{y\in V}
b(x_0, y)
(u(x_0) - u(y)) + \frac{c(x_0)}{m(x_0)} u(x_0) + \alpha u (x_0)$$
there must exist an $x_1$ connected to $x_0$ with $u(x_1)< u(x_0)$. Continuing in this way, we obtain a sequence $(x_n)$ of connected
 points with $u(x_n) < u(x_0)<0$. Combining this with $(A)$, we obtain a
 contradiction to $u\in \ell^p (V,m)$.
\end{proof}

Let us note the following consequence of the previous minimum principle.

\begin{lemma} \label{vanishingonlp} (Uniqueness of solutions on $\ell^p$)
  Assume $(A)$. Let $\alpha>0$, $p\in [1,\infty)$
  and $u\in \ell^p (V,m)$ with $(\widetilde{L} + \alpha) u =0$ be
  given. Then, $u\equiv 0$.
\end{lemma}
\begin{proof} Both $u$ and $-u$ satisfy the assumptions of the previous
  proposition. Thus, $u\equiv 0$.
\end{proof}

\noindent\textbf{Remark.} The situation for $p=\infty$ is substantially more complicated as can be seen  by  (part (ii) of) our first theorem.

\smallskip

This lemma allows us the describe the generators whenever $(A)$ holds.

\begin{proof}[Proof of Theorem~\ref{generator}]
Define
$$\widetilde{D}_p:= \{ u\in \ell^p (V,m) : \widetilde{L} u \in
\ell^p (V,m)\}.$$
By Theorem~\ref{obermenge}, we already know   $L_p f = \widetilde{L} f$ for any
  $f\in D(L_p)$. It remains to  show $ \widetilde{D}_p\subseteq D(L_p)$: Let $f\in \widetilde{D}_p$
be given. Then, $g:= (\widetilde{L} + \alpha) f$ belongs to $\ell^p (V,m)$.
Thus, $u:= (L_p + \alpha)^{-1} g$ belongs to $D (L_p)$. Now,
as shown above, see Lemma~\ref{lzwei},  $u$ solves $(\widetilde{L}+ \alpha) u = g$. Moreover, $f$
also solves this equation. Thus, by the uniqueness of solutions given in Lemma
\ref {vanishingonlp}, we infer $f= u$ and $f$ belongs to $D (L_p)$.
 This finishes the proof.
 \end{proof}

We now turn to a study of essential selfadjointness on $C_c (V)$. Clearly,
the question of essential selfadjointness on $C_c (V)$ only makes sense if
$\widetilde{L} C_c (V)\subseteq \ell^2 (V,m)$. In this context, we have the
following result:

\begin{prop}\label{nutzen} Let $(V,m)$ be given and $(b,c)$ a weighted graph over $V$. Then, the
  following assertions are equivalent:

\begin{itemize}

\item[(i)] $\widetilde{L} C_c (V)\subseteq \ell^2 (V,m)$.

\item[(ii)] For any $x\in V$, the function $V \longrightarrow [0,\infty)$,
  $y\mapsto b(x,y)/m(y)$, belongs to $\ell^2 (V,m)$.
\end{itemize}
In this case, any $u\in \ell^2 (V,m)$ belongs to the set $\widetilde{F}$ of \eqref{ftilde}  on which
$\widetilde{L}$ is defined   and the three sums
$$ \sum_{x\in V} u(x) \widetilde{L} v (x) m(x),\quad \sum_{x\in V} \widetilde{L} u
(x)  v(x) m(x)$$
and $$ \frac{1}{2} \sum_{x,y\in V} b(x,y) (u(x) - u(y) (v(x) - v(y)) + \sum_{x\in V}c(x) u(x) v(x)$$
converge absolutely and agree
%$$ \sum_{x\in V} u(x) \widetilde{L} v (x) m(x)= \sum_{x\in V} \widetilde{L} u
%(x)  v(x) m(x)= \sum_{x,y\in V} b(x,y) (u(x) - u(y) (v(x) - v(y))$$
for all $u\in \ell^2 (V,m)$ and $v\in C_c (V)$.
\end{prop}
\begin{proof} Without loss of generality we assume $c\equiv 0$.  For any $x\in V$ define  $\delta_x : V\longrightarrow \RR$ by
$\delta_x (y) = 1$ if $x=y$ and $\delta_x (y) = 0$ if $x\neq y$.

Obviously,  (i) is equivalent to $\widetilde{L}\delta_x\in \ell^2 (V,m)$ for
all $x\in V$. This latter condition can easily be seen to be equivalent to
(ii).  This shows the stated equivalence.

Let $u\in \ell^2 (V,m)$ be given. Then, for any $x\in V$,
Cauchy-Schwarz inequality and (ii)  give
$$(*)\:\;\:  \sum_{y\in V} |b(x,y) u(y)| \leq \left(\sum_{y\in V} \frac{b (x,y)^2}{m(y)}
\right)^{1/2} \left(\sum_{y\in V} u(y)^2  m(y) \right)^{1/2} <\infty.$$
Thus, $u$ belongs to $\widetilde{F}$. To show the
statement on the sums, it suffices to consider $u\in \ell^2 (V,m)$ and $v=
\delta_z$, for $z\in V$ arbitrary. In this case, the desired statements can
easily be reduced to the question of absolute convergence of
$$\sum_{x,y\in V} b(x,y) u(x) \delta_z (x) \;\:\mbox{and}\:\; \sum_{x,y\in V} b(x,y)
u(x) \delta_z (y).$$
This absolute convergence in turn is  shown in $(*)$.
\end{proof}

\begin{proof}[Proof of Theorem~\ref{essential}]
As $\widetilde{L} C_c (V)\subseteq \ell^2 (V,m)$, we can define the minimal
operator $L_{\min }  $ to be the restriction of $\widetilde{L}$ to
$$ D(L_{\min} ):= C_c (V)$$
 and the maximal operator $L_{\max}$ to be the restriction of $\widetilde{L} $ to
$$D (L_{\max}) := \{u\in \ell^2 (V,m) : \widetilde{L} u \in\ell^2 (V,m)\}.$$
The previous proposition gives
$$\langle u, L_{\min} v\rangle = Q^{\comp}_{b,c} (u,v)$$
for all $u,v\in C_c (V)$. This extends to give
$$ \langle u, L_{\min} v\rangle = Q_{b,c,m} (u,v)$$
for all $u\in D(Q)$ and $v\in C_c (V)$. Thus, $L_{\min}$ is a restriction of $L$ in this case.\\
Moreover, the previous proposition gives also
$$ \langle u, L_{\min} v\rangle = \sum_{x\in V} \widetilde{L}
u(x) v(x)  m(x)$$
for all $v\in C_c (V)$ and $u\in \ell^2 (V,m)$. Thus,
$$ L_{\min}^\ast  = L_{\max}.$$
Thus, essential selfadjointness of $L_{\min}$ is equivalent to selfadjointness of $L_{\max}$. This in turn is equivalent to  $L= L_{\max}$ (as we have
$ L \subseteq L_{\max}$ by Theorem~\ref{generator}). As  $(A)$ and Theorem~\ref{generator} yield $D(L) = \{u\in \ell^2 (V,m) : \widetilde{L} u \in\ell^2 (V,m)\}$, we infer $L=L_{\max}$ and essential selfadjointness of the restriction of $L$ to $C_c (V)$ ($= L_{\min}$) follows.\\
It remains to show the statement on the form. Let $Q^{\max}$ be the
maximal form, i.e., $$ Q^{\max}(u) =\frac{1}{2}\sum_{x,y\in V} b(x,y) (u(x) - u(y))^2+\sum_{x\in V}c(x)u(x)^2$$
for all $u\in \ell^2 (V,m)$ and $L_{Q^{\max}}$ the associated operator.
 Then, another application of the previous
proposition shows
\begin{eqnarray*}
\lefteqn{\sum_{x\in V} \widetilde{L} u (x) v(x) m(x)}\\
&=& \frac{1}{2}\sum_{x,y\in V} b(x,y) (u(x)- u(y))(v(x) - v(y))+\sum_{x\in V}c(x)u(x)v(x)\\
&=&Q^{\max} (u,v)\\
& = &\langle L_{Q^{\max}} u, v\rangle
\end{eqnarray*}
for all $u\in D(L_{Q^{\max}})$ and $v\in C_c (V)$. This gives that the
self-adjoint operator $L_{Q^{\max}}$ associated to $Q^{\max}$
satisfies
$$  L_{Q^{\max}} u = \widetilde{L} u = L u$$
for all $u\in D ( L_{Q^{\max}})$. Thus,
$$  L_{Q^{\max}} \subseteq L.$$
As $L$ is selfadjoint, we infer  $L_{Q^{\max}} = L$ and the
statement on the form follows.
 \end{proof}

\section{Some counterexamples}\label{Counterexamples}
In this section, we first discuss an  example showing that without condition $(A)$ Theorem~\ref{generator} and Theorem~\ref{essential}  fail in general. We then present an example of a non-regular Dirichlet form on a weighted graph. Note that the choice of the measure plays  a crucial role here.

\medskip

\textbf{Example for failure of Theorem~\ref{generator} and \ref{essential}  without assumption $(A)$}.\\
Let $V\equiv \Z$.  Let (at first) every point of $\Z$ have measure $1$. Consider the bounded operator
$$\Lp:\ell^2(\Z)\longrightarrow \ell^2(\Z), \, (\Lp \ph)(x)=-\ph(x-1)+2\ph(x)-\ph(x+1).$$
It   corresponds to the  Dirichlet form $Q_{b,0,1}$ with $b(x,y) = 1$ whenever $|x-y|=1$.
A direct calculation shows that the function
$$u: V\longrightarrow \RR, \;u(x):= e^{\lambda x}$$
 is a positive solution to the equation $(\ow\Lp+\al) u=0$ for $\al=e^{-\lambda}+e^\lambda -2$. Obviously, we have $\al\geq0$ for all real  $\lambda$.
Now let $w\in\ell^1(\Z)$, $w>0$ and define the measure
$$m: V\longrightarrow (0,\infty), \; m(x) :=\min\{1,\frac{w (x)}{u^2 (x)}\}$$
and the killing term
$$c : V\longrightarrow [0,\infty), \: c(x):=\max\{0,\frac{u^2 (x)}{w(x)}-1 \}\al m(x).$$
 By construction, $u$ then  belongs to $\ell^2(\Z,m)$ and
 $$ - \frac{\alpha}{m} + \frac{c}{m} + \alpha \equiv 0.$$
Let $\ow L$ be defined by
 $$\ow L v (x):=\frac{1}{m(x)} \sum_{y\in V} b(x,y) (v(x) - v(y)) + \frac{c(x)}{m(x)} v(x),$$
i.e., in a formal sense $\ow L=\frac{1}{m}(\ow \Lp+c)$. Then, the restriction $L_{\max}$ of $\ow L$ to
 $$\{v \in\ell^2(\Z,m):\ow L v\in \ell^2(\Z,m)\}$$
  has the eigenvalue $-\al<0$ since
$$(L_{\max}+\al)u(x)=\ab{-\frac{\al}{m}+\frac{c}{m} +\al}u(x)=0.$$
Consider now the
operator $L$ associated with the Dirichlet form $Q_{b,c,m}$ on $\ell^2(\Z,m)$.  Of course, $L$ is a positive operator and therefore can not have a negative eigenvalue. Moreover, by the results of the previous section,
this operator is a restriction of $\ow L$. This implies that $u$ can not belong to $ D(L)$ and therefore $D(L)\neq D(L_{\max})$. Thus, the domain of definition $D(L)$ is not given by Theorem~\ref{generator}.  In this case, the restriction of $\widetilde{L}$ to $C_c (V)$   is not essentially  self-adjoint (as the proof of Theorem~\ref{essential} showed that otherwise $L=L_{\max}$).

\medskip

\textbf{Example of a non-regular Dirichlet form on $V$}\\
We consider connected graphs $(V,b,c)$ with $c\equiv 0$ and $b(x,y)\in \{0,1\}$ for all $x,y\in V$. As discussed by
Dodziuk-Kendall \cite{DK} (see \cite{Dod,Kel} as well) in the context of isoperimetric inequalities, any  such graph with positive Cheeger constant $\alpha>0$  has the property that
$$\frac{1}{2} \sum_{x,y\in V} b(x,y) (\varphi (x) - \varphi (y))^2  \geq \frac{\alpha^2}{2} \sum_{x\in V} d(x) \varphi (x)^2$$
for all $\varphi \in C_c (V)$, where $d(x) = \sum_{y\in V} b(x,y)$.  Let now such a graph be given. Fix an arbitrary $x_0\in V$. Choose a measure $m$ with support  $V$ and  $m(V) =1$. Thus, the constant function $1$ belongs to $\ell^2 (V,m)$.  Define the form $Q$ by
$$ Q(u) := \frac{1}{2}\sum_{x,y\in V} b(x,y) (u(x) - u(y))^2$$
for all $u\in \ell^2 (V,m)$ for which $Q(u)$ is finite. Obviously, $Q$ is a Dirichlet form and the constant function $1$ satisfies $Q(1) =0$. Let now $\varphi_n$ be any sequence in $C_c (V)$ converging to $1$ in $\ell^2 (V,m)$. Then, $\varphi_n (x_0)$ converges to $1$. In particular,
$$Q(\varphi_n) \geq  \frac{\alpha^2}{2} d(x_0) \varphi_n (x_0)^2\to \frac{\alpha^2}{2} d(x_0)>0,\quad n\to \infty.$$
Thus, $Q(\varphi_n)$ does not converge to $0   = Q(1)$. Hence, $Q$ is not regular.

\section{The heat equation on $\ell^\infty$} \label{heat}
In this section we consider a weighted  graph $(b,c)$ over the measure space $(V,m)$   with associated formal operator $\widetilde{L}$.

\smallskip

A function $N : [0,\infty)\times V\longrightarrow \RR$ is called a solution of
the heat equation if for each $x\in V$ the function $t\mapsto N_t (x)$ is
continuous on $[0,\infty)$ and differentiable on $(0,\infty)$ and for each
$t>0$ the function $N_t$ belongs to the domain of $\widetilde{L}$ and the
equality
$$ \frac{d}{dt} N_t (x) =  - \widetilde{L} N_t (x)$$
holds for all $t>0$ and $x\in V$.  For a bounded solution $N$ continuity of $[0,\infty)\longrightarrow \R$, $t\mapsto \widetilde{L} N_t (x)$, can easily be seen (for each fixed $x\in V$). For such $N$  validity of  $ \frac{d}{dt} N_t (x) =  - \widetilde{L} N_t (x)$ for $t>0$ then extends
 automatically  to  $t=0$, i.e., $t\mapsto N_t (x)$ is
differentiable on $[0,\infty)$ and
 $ \frac{d}{dt} N_t (x) =  - \widetilde{L} N_t (x)$ holds for any $t\geq
 0$.

\smallskip

The following theorem is essentially a standard result in the theory of semigroups. In the situation of special graphs it has been shown in \cite{Web, Woj}. For completeness reason we give a proof in our situation as well.

\begin{theorem}\label{solution}
 Let $L$ be a self-adjoint restriction of $\widetilde{L}$, which is
 the generator of a Dirichlet form on $\ell^2 (V,m)$.
 Let $v$ be a bounded function on $V$ and define  $N: [0,\infty)\times
  V\longrightarrow \RR$ by $N_t (x) := e^{-t L} v (x)$. Then, the function $N(x) : [0,\infty)\longrightarrow \RR$, $t\mapsto
  N_t (x)$, is differentiable and satisfies
$$ \frac{d}{d t} N_t (x) = - \widetilde{L} N_t (x)$$
$\mbox{for all $x\in V$ and $t\geq 0$.}$

\end{theorem}
\begin{proof}% Without loss of generality we can assume $m\equiv 1$.
As $v$ is bounded,  continuity of $t\mapsto N_t (x)$ follows from general principles on weak $\ell^1$-$\ell^\infty$  continuity of the semigroup on $\ell^\infty
  (V,m)$, see e.g. \cite{Dav1}. It remains to show differentiability and the validity of the equation.

As discussed already, it suffices to consider $t>0$.
After decomposing $v$ into positive and negative part, we can assume without
  loss of generality that  $v$ is non-negative.

Let $(K_n)$ be sequence of finite increasing  subsets of $V$ with $\bigcup K_n =
  V$. Let $v_n$ be the function on $V$ which agrees with $v$ on $K_n$ and
  equals zero elsewhere. Thus, $v_n$ belongs to $\ell^2 (V,m)$ and we can
  consider $e^{- t L} v_n$ for any $n\in
  \N$.  For each fixed $x\in V$ the function  $t\mapsto e^{- t L} v_n (x)$ converges monotonously to $t\mapsto N_t(x)$ (by definition of the
  semigroup on $\ell^\infty$).  As $t\mapsto N_t(x)$
  is  continuous, this convergence is even uniformly on compact subintervals
  of $(0,\infty)$.
Moreover,  standard $\ell^2$ theory shows that $N_n=e^{- t L} v_n $  satisfies $\frac{d}{dt} N_n  (x) = - L N_n (x) $ for all $t>0$ and $x\in V$.    By
  assumption $L$ is a restriction of $\widetilde{L}$ and we infer
$$ \frac{d}{dt} N_n (x) = - \frac{1}{m(x)} \sum_{y\in V} b(x,y) (N_n (x) - N_n (y)) -  \frac{c(x)}{m(x)} N_n (x)$$
for all $x\in V$ and $t>0$. This equality together with the
uniform convergence  on compact
intervals  in $(0,\infty)$ and  the summability of the
$b(x,y)$ in $y$ gives  uniform convergence of the $\frac{d}{dt} N_n (x) $ on compact
intervals. Hence, $t\to N_t (x)$ is differentiable on $(0,\infty)$ and
satisfies the desired equation.
\end{proof}

The next lemma shows how solutions of the heat equation to vanishing initial conditions give rise to non-trivial bounded solutions of $(\ow L+\al)v=0$, (see (vi)$\Rrightarrow$(ii) Theorem~\ref{main0}).

\begin{lemma} \label{heattosubharmonisch} Let $N$ be a bounded solution of $\frac{d}{d t}N + \widetilde{L} N
  =0$, $N_0 \equiv 0$. Then, for arbitrary  $\alpha >0$,   the function   $v:=\int_0^\infty  e^{- t \alpha}  N_t dt$ solves
  $(\widetilde{L}+ \alpha) v =0$.
\end{lemma}
\begin{proof}
This follows by a short calculation: By boundedness of $N$ and  $\sum_y b(x,y) <\infty$, we can interchange two limits to  obtain
$$ \widetilde{L} v (x) = \lim_{T\to \infty} \int_0^T  e^{- t \alpha}
\widetilde{L} N_t (x) dt.$$
Using that  $N$ solves the heat equation  and partial integration we find
\begin{eqnarray*}
\widetilde{L} v (x) &=& \lim_{T\to \infty}  \int_0^T  e^{- t \alpha}
  (- \frac{d}{d t}N_t (x)) dt\\
&=& \lim_{T\to \infty} \left(\left .- e^{- t\alpha} N_t (x)\right|_0^T - \int_0^T \alpha  e^{-
 t\alpha} N_t (x) dt\right)\\
&=& - \alpha \int_0^\infty  e^{- t\alpha} N_t (x) dt\\
&=& - \alpha v(x).
\end{eqnarray*}
Here, we used boundedness of $N$ and $N_0 =0$ to get rid of the boundary terms
after the partial integration.
\end{proof}

\section{Extended semigroups and resolvents} \label{Extended}
 We are now going to extend the resolvents/semigroups to a larger class of
functions. To do so, we note that for  a function $ f$ on $V$ with $f\geq 0$   the functions
 $g\in C_c (V)$ with  $0\leq g \leq f$ form a net with respect to the natural
 ordering  $g\prec h$ whenever  $g \leq h$.  Limits along this net will be
 denoted by  $\lim_{g \prec f}$. As the resolvents and semigroups are
 positivity preserving, for  $f\geq 0$, $\alpha >0$, $t>0$,  we can define
 the functions $(L +\alpha)^{-1} f : V\longrightarrow [0,\infty]  $ and $
 e^{- tL} f : V \longrightarrow [0,\infty]$ by
\begin{align*}
(L +\alpha)^{-1} f (x) &:=  \lim_{g \prec f} (L + \alpha)^{-1} g(x)\\
e^{- tL } f(x) &:=\lim_{g \prec f} e^{- tL } g (x).
\end{align*}

In fact, as
we are in a discrete setting, the operators have kernels, i.e.,  for any $t\geq
0$ there exists a unique function
$$e^{- tL} :  V\times V \longrightarrow
[0,\infty)\;\: \mbox{with}\;\:   e^{- tL } f(x) = \sum_{y\in V} e^{- tL} (x,y)
f(y)$$
for any $f\geq 0$ (and similarly for the resolvent).  It is not hard to see that for
functions in $\ell^\infty (V)$, these definitions are consistent with our
earlier definitions.

\begin{theorem} \label{resolvents} (Properties of extended resolvents and semigroups)
Let $\alpha>0$ be given. Let $f$ be a non-negative function on $V$.
\begin{itemize}
 \item[(a)] Let  $K_n$ be an increasing  sequence of finite subsets of $V$ with
  $\bigcup K_n = V$. Let  $f_n$ be the restriction of  $f$ to  $K_n$, and $u_n
  :=(L_{K_n}^{(D)} + \alpha)^{-1} f_n$. Then, $u_n$ converges pointwise
  monotonously to $ (L +\alpha)^{-1} f$.

 \item[(b)]  The following statement are equivalent:

\begin{itemize}
\item[(i)] There exists a non-negative  $l : V\longrightarrow [0,\infty)$ with
  $(\widetilde{L} + \alpha) l \geq f$.
\item[(ii)] $(L + \alpha)^{-1} f (x) $ is finite for any $x\in V$.
\end{itemize}
In this case $u:=(L + \alpha)^{-1} f$ is the smallest non-negative function
$l$ with $(\widetilde{L} + \alpha) l \geq f$ and it satisfies $(\widetilde{L}+\alpha) u = f$.

\item[(c)]  For all $x\in V$
$$(L + \alpha)^{-1} f (x) = \int_0^\infty  e^{- t \alpha} e^{-t L} f (x) dt.$$
\end{itemize}
\end{theorem}

\noindent\textbf{Remark.} Note that the functions in (a) and (c) are allowed to take the value $\infty$. Statement (a) is an extension of Proposition~\ref{convergence} to non-negative functions.

\begin{proof} Throughout the proof we let $u$ denote the function $(L+\alpha)^{-1} f$.

\smallskip

(a) Let $x\in V$ be given. By domain  monotonicity $u_n (x) = (L_{K_n}^{(D)} + \alpha)^{-1} f_n  (x)$ is
   increasing. Moreover,  again by domain monotonicity and
$f_n \leq f$ we have
$$u_n (x) =  (L_{K_n}^{(D)} + \alpha)^{-1} f_n (x)  \leq (L +
\alpha)^{-1} f_n (x)  \leq (L + \alpha)^{-1} f (x)  = u (x)$$
for all  $n$.  It remains to show the 'converse' inequality. We consider two
cases.\\
\textit{Case 1. $u(x) <\infty$.} Let $\varepsilon >0$ be given. By definition of the extended resolvents there exists then $g\in C_c (V)$ with $0\leq g \leq f$
and
$$ u(x) - \varepsilon \leq (L + \alpha)^{-1} g (x). $$
As $g$ has compact support, we can assume without loss of generality that the
support of $g$ is contained in $K_n$ for all $n$.
By convergence of resolvents, we conclude
$$ (L + \alpha)^{-1} g (x)  -\varepsilon \leq (L_{K_n}^{(D)} + \alpha)^{-1} g
(x) $$
for all sufficiently large $n$.
Thus, for such $n$ we find
$$u(x) - 2\varepsilon  \leq (L_{K_n}^{(D)} + \alpha)^{-1} g (x) .$$
By $g\leq f$ and $\supp g \subseteq K_1$, we have $g\leq f_n$ for all $n$.
Thus, the last inequality gives
$$ u(x) - 2 \varepsilon \leq (L_{K_n}^{(D)} + \alpha)^{-1} f_n (x) = u_n (x).$$
This finishes the considerations for this case.\\
\textit{Case 2. $u(x) = \infty$.} Let $\kappa >0$ be arbitrary. By definition of the extended resolvents there exists then $g\in C_c (V)$ with $0\leq g \leq f$
and
$$ \kappa  \leq (L + \alpha)^{-1} g (x). $$
Now, we can continue as in Case 1 to obtain
$$ \kappa - \varepsilon \leq (L_{K_n}^{(D)} + \alpha)^{-1} f_n (x) = u_n (x)$$
for all sufficiently large $n$. As $\kappa >0$ is arbitrary the statement
follows.

\smallskip

(b) We first show (ii)$\Longrightarrow$(i): Recall that $u=(L+ \alpha)^{-1}
f$ and consider  $g\in C_c (V)$ with
$0 \leq g \leq f$. Then, by Lemma~\ref{lzwei}, $u_g := (L + \alpha)^{-1} g$ solves
$$ (\widetilde{L} +\alpha) u_g = g.$$
Taking monotone limits on both sides  and using the finiteness assumption (ii), we obtain
$$ (\widetilde{L} + \alpha) u = f.$$
This shows (i) (with $l=u$).\\
We next show (i)$\Longrightarrow$(ii): Let  $l \geq 0$ satisfy  $(\widetilde{L} +
\alpha) l \geq f$. Let $(K_n)$ be an increasing sequence of finite subsets of
$V$ as in (a) and let $f_n$ be the restriction of $f$ to $K_n$.   Extend
$u_n := (L_{K_n}^{(D)} + \alpha)^{-1} f_n$ by zero to all of $V$.
Then,   $w_n:= l - u_n$ satisfies:
\begin{itemize}
 \item  $w_n= l \geq 0$ on  $K_n^c$.
\item  The negative part of $w_n$ attains its minimum on $K_n$ (as $K_n$ is finite).
  \item $(\widetilde{L} + \alpha) w_n = (\widetilde{L} + \alpha) l - (\widetilde{L} + \alpha)
  u_n \geq f - f =0$ on  $K_n$.
\end{itemize}
The minimum principle, Theorem \ref{minimum-principle},  then gives
$$  w_n= l - u_n\geq 0.$$
As  $n$ is arbitrary and  $u_n$ converges to $u$ by part  (a), we find that
$u\leq l$ is finite. This finishes the proof of  the equivalence statement of (b). The last statements of (b) have already been shown along the  proofs of (i)$\Longrightarrow$ (ii) and (ii)$\Longrightarrow $(i).

\smallskip

(c) For  $g\in C_c (V)$ with  $0\leq g \leq f$  the equation
$$ (L + \alpha)^{-1} g = \int_0^\infty  e^{- t \alpha} e^{-t L} g dt$$
 holds by standard theory on semigroups. Now, (c) follows by taking monotone
 limits on both sides.
\end{proof}

There is a special function $v$ to which our considerations can be applied:

\begin{prop}\label{bound} For  any $\alpha >0$  we have the estimate
$$ 0 \leq (L + \alpha)^{-1} (\alpha 1 + \frac{c}{m}) \leq 1.$$
\end{prop}

\noindent\textbf{Remark.} Let us stress that  $c/m$ is not assumed to be bounded.

\begin{proof} As $\alpha 1 + c/m \geq 0$, we have $ 0 \leq (L +
  \alpha)^{-1} (\alpha 1 + c/m) $. Moreover, we obviously have
$$ (\widetilde{L}+\alpha) 1 = \alpha 1 + \frac{c}{m}.$$
Thus, (b) of the previous theorem shows $(L + \alpha)^{-1} (\alpha 1 + c/m)
\leq 1$.
 \end{proof}

We will also  need the following consequence of the proposition.

\begin{prop}\label{propertyofS} Let $(V,b,c)$ be a weighted  graph and define $S: V\longrightarrow [0,\infty]$ by
$$ S(x) :=\int_0^\infty   ({e^{-s L} \frac{c }{m}})(x)  ds.$$
Then,  $S$ satisfies $0\leq S \leq 1$ and $\widetilde{L} S = c/m$.
\end{prop}
\begin{proof}  For $g\in C_c (V)$ with $0\leq g \leq c/m$ and $\alpha>0$, we define $S_{g,\alpha}$ by $S_{g,\alpha}: = \int_0^\infty e^{- \alpha s}e^{- sL} g (x) ds$ and $S_g$ by $S_g=\lim_{\alpha \to0} S_{g,\alpha}$.
Then,
$$ S_{g,\alpha} = (L + \alpha)^{-1} g,\quad \mbox{i.e.,}\quad (\widetilde L+\alpha) S_{g,\alpha} = g.$$
By $g\leq \alpha 1 + c/m$ for any $\alpha >0$ and Proposition~\ref{bound}, we
have
$$S_{g,\alpha} = (L + \alpha)^{-1} g (x) \leq (L + \alpha)^{-1} (\alpha 1
+ \frac{c}{m}) (x) \leq 1.$$
As $S_g$ is the monotone limit of the $S_{g,\alpha}$, this shows that $S_g$ is
bounded by $1$. Moreover, using the uniform bound on the $S_{g,\alpha}$ and taking the limit $\alpha \to 0$  in
$$ (\widetilde{L} + \alpha) S_{g,\alpha} =  g,$$
we find
$$ \widetilde{L} S_g = g\geq 0.$$
As $S= \lim\limits_{g\prec c/m} S_g$ and the $S_g$ are uniformly bounded, we obtain the
statement.
\end{proof}

\begin{lemma}  \label{excessiv}  Let  $u\geq 0$ be given. Then,  the following assertions are
  equivalent:
\begin{itemize}
\item[(i)] $e^{-t L} u \leq u$ for all $t>0$.
\item[(ii)] $(L + \alpha)^{-1}u \leq \frac{1}{\alpha}u$ for all $\alpha
  >0$.
\end{itemize}
Any   $u\geq 0$ with
$\widetilde{L} u \geq 0$ satisfies these equivalent conditions.
\end{lemma}
\begin{proof} The implication (i)$\Rightarrow$(ii) follows easily from
$(L + \alpha)^{-1} = \int_0^\infty  e^{- t \alpha} e^{-t L} dt$.  Similarly,
the implication (ii)$ \Rightarrow $(i) follows by a limiting argument from the standard
$$ e^{- tL} f = \lim_{n\to \infty} \left(\frac{t}{n} \ab{L + \frac{n}{t}}\right)^{-n} f$$
for $f\in\ell^{2}(V,m)$.
As for the last statement, we note that $ \widetilde{L} u \geq 0$ implies
$$ \frac{1}{\alpha} (\widetilde{L} +\alpha) u \geq u.$$
By (b) of Theorem~\ref{resolvents} the desired  statement (ii) follows.
\end{proof}

\section{Characterization of stochastic completeness} \label{Characterisation}
In this section, we can finally characterize stochastic completeness.
We begin by introducing the crucial quantity in our studies.

\begin{lemma}\label{propertyofM}
  Let $(V,b,c)$ be a weighted graph and $m$ a measure on $V$ with full
  support.  Then, the function $M : [0,\infty)\times
  V\longrightarrow [0,\infty]$ defined by
$$ M_t (x) := e^{- tL} 1 (x) + \int_0^t  ({e^{-s L} \frac{c}{m}})(x) ds$$
satisfies     $0\leq M_s\leq M_t  \leq 1$ for all $s\geq t\geq 0$ and, for each $x\in V$, the map $t\mapsto M_t (x)$ is differentiable and satisfies  $\frac{d}{d t}M_t (x) + \widetilde{L} M_t(x) =  c (x)/m(x)$.
\end{lemma}

\noindent\textbf{Remark.} We can give an interpretation of $M$ in terms of a diffusion
process on $V$ as follows. For $x\in V$, let $\delta_x$ be the characteristic function of $\{x\}$.
A diffusion
on $V$ starting in  $x$ with normalized measure is then given by $\frac{1}{m(x)} \delta_x$ at time $t=0$. It will yield to the amount of heat
$$\langle e^{-t L} \frac{\delta_x}{m(x)}, 1\rangle = \langle \frac{\delta_x}{m(x)}, e^{-t L} 1\rangle = \sum_{y\in V} e^{-tL}(x,y)$$
within $V$ at the time $t$. Moreover, at each time $s$ the
rate of heat killed at $y$ by the  killing term $c$ is given by $e^{-sL}(x,y)c(y)/m(y)$. The total amount of heat killed at $y$ until  the time $t$ is then given by $\int_0^t   e^{ -sL} (x,y)c(y)/m(y) ds$. The total amount of
heat killed at all vertices by $c$ till the time $t$ is accordingly given by
$$ \sum_{y\in V}  \int_0^t   e^{- sL} (x,y)\frac{c(y)}{m(y)} ds = \int_0^t
\sum_{y\in V} e^{- sL} (x,y)  \frac{c(y)}{m(y)} ds =  \int_0^t (e^{- sL} \frac{c}{m}) (x)
ds.  $$
This means that $M$ measures the amount of heat at time $t$ which has not been transferred to the boundary of $V$.

\begin{proof} By definition we have $M\geq 0$. By Proposition~\ref{propertyofS},  $ S (x) = \int_0^\infty  (e^{- sL} c /m)(x) ds$
is finite and we can therefore calculate
$$  \int_0^t (e^{- sL} \frac{c }{m})(x) ds = S (x) - \int_t^\infty (e^{- sL} \frac{c }{m})(x) ds
= S(x) - e^{- t L} S (x),$$
where the last statement follows by taking monotone limits along the net of
$g\in C_c (V)$ with $0\leq g \leq c/m$.
Thus,
$$M_t = e^{-t L} 1 + S - e^{- t L} S= S + e^{-tL} (1-S).$$
From this equality the desired statements follow easily:
 By  Proposition  \ref{propertyofS}, we have   $1 - S\geq 0$ and
$\widetilde{L} (1- S) = \widetilde{L} 1 - \widetilde{L} S = c/m - c/m= 0$. Lemma \ref{excessiv} then yields
$$ e^{- s L} ( 1 - S) \leq  e^{- t L} ( 1 - S)   \leq   1 - S$$
for all $s\geq t\geq 0$. Plugging this into the formula for $M_t$ gives, for all $0\leq t \leq s$,
$$ 0 \leq M_s \leq M_t\leq 1.$$
Moreover, as $S$ and the constant function $1$ are bounded, we can apply Theorem~\ref{solution} to  $M_t = e^{-t L}( 1-S)  + S$ to infer
that $t\mapsto M (x)$ is differentiable with
$$ \frac{d}{d t}M_t (x) = - \widetilde{L}  e^{-t L} 1 (x)  +
\widetilde{L}e^{- t L} S (x) = - \widetilde{L} M_t (x) +
\widetilde{L} S (x) =  - \widetilde{L} M_t (x)  + \frac{c(x)}{m(x)},$$
where we used $\widetilde{L} S = c/m$ from from Proposition \ref{propertyofS}.
\end{proof}

We now show that integration over $M$ yields a resolvent.

\begin{lemma}\label{integration} $ (L + \alpha)^{-1} (\alpha 1 + c/m) (x)  = \int_0^\infty
  \alpha e^{- t\alpha}   M_t (x) ds.$
\end{lemma}
\begin{proof}  As shown in (c) of Theorem~\ref{resolvents} we have
$$ (L + \alpha)^{-1} (\alpha 1 + \frac{c}{m}) (x) = \int_0^\infty
\alpha e^{- t\alpha}   e^{-t L} ( 1 + \frac{c}{\alpha m}) (x) dt. $$
Thus, it suffices to show that
$$ \int_0^\infty e^{- t\alpha}   (e^{-t L} \frac{c}{m})(x) dt = \int_0^\infty \alpha
e^{- t\alpha} \ab{\int_0^t (e^{-s L}\frac{c}{m})(x)  ds} dt.$$
This follows by partial integration applied to each (non-negative) term of the
sum $$ (e^{-t L}\frac{c}{m}) (x) = \sum_{y\in V} e^{-t L} (x,y) \frac{c(y)}{m(y)}.$$
This finishes the proof.
\end{proof}

\noindent\textbf{Remark.} Let us stress that the care taken with monotone convergence in the above arguments is quite necessary. For example one might think that $1 = (L+ \alpha)^{-1} (\ow L+\alpha) \, 1$. Combined with the previous lemma, this would lead to $1 = (L+\alpha)^{-1} (\alpha 1 + c/m)= \int_0^\infty \alpha e^{-t\alpha} M_t dt.$ However, the phenomenon we study  is exactly that the integral can be strictly smaller than $1$!

\smallskip

After these preparations we now prove our first  main result. Recall that we defined
$$ w = \int_0^\infty \alpha e^{-t \alpha } ( 1 - M_t) dt.$$

\begin{proof}[Proof of Theorem~\ref{main0}.]  As $\int_0^\infty \alpha e^{- t\alpha} dt =1$,    Lemma~\ref{integration} gives  $w = 1 - (L + \alpha)^{-1} (\alpha 1 + c/m) $. Thus, $w$
  solves $(\widetilde{L} + \alpha) w = 0$. Moreover, the minimality
  properties of the extended resolvent yield the maximality property of
  $w$.  More precisely, let $l$ be any non-negative function bounded by $1$ with $(\widetilde{L} + \alpha)l\leq  0$. Then, $ 1-l$ is non-negative and satisfies
  $$(\widetilde{L} +\alpha) (1-l) = \alpha 1+\frac{c}{m} - (\widetilde{L} +\alpha)l \geq \alpha 1+\frac{c}{m}.$$
The minimality property of  $1 - w = (L+\alpha)^{-1} (\alpha 1 + c/m)$ then gives $ 1-w  \leq  1-l$, and the desired inequality $l\leq w$ follows.

\smallskip

It remains to show the equivalence statements.\\
(v)$\Longrightarrow$(iv): This is clear as $0\leq M_t \leq 1$ and $M$ is continuous.\\
(iv)$\Longrightarrow$(iii): This is clear from the properties of $w$ shown above.\\
(iii)$\Longrightarrow$(ii):  This is clear.\\
(ii)$\Longrightarrow $ (i):  Let $l^+$ be the positive part of $l$, i.e.,  $l^+ (x) = l(x)$ if $l(x)>0$ and $l^+ (x) =0$ otherwise. If $l^+$ is trivial, the function $- l$ is a nontrivial, non-negative bounded  solution and (i) follows. Otherwise a direct calculation shows that $l^+$ is a nontrivial subsolution. Obviously, $l^+$ is non-negative and bounded. \\
(i)$\Longrightarrow$ (v): If there exists a nontrivial non-negative  subsolution, then $w$
as the largest subsolution  must be nontrivial. Hence, there must exist $t>0$
and $x\in V$ with $M_t (x) <1$.\\
(v) $\Longrightarrow$ (vi):  Lemma \ref{propertyofM} gives that   $N:= 1 - M$
satisfies $N_0 =0$ and $\frac{d}{d t}N + \widetilde{L} N =0$. This gives the desired
implication.\\
(vi)$\Longrightarrow$ (i): This is a direct consequence of Lemma
\ref{heattosubharmonisch}.
\end{proof}

\section{Stochastically complete graphs with incomplete subgraphs} \label{Stochastically}
In this section we prove Theorem~\ref{main1b}  and Theorem~\ref{main1}. For Theorem~\ref{main1} the basic idea is to attach graphs satisfying $\SC$ to each vertex of a graph (with  $\SI$) such that the resulting graph will satisfy $\SC$. As adding a potential to a graph can be interpreted as adding edges to infinity, the proof of Theorem~\ref{main1b} can be seen as a variant of the proof of Theorem~\ref{main1}.

\medskip

The graphs we attach will be the following. Let
$(N,b_N,0)$ the graph with vertex set
$N=\{0,1,2,\ldots\}$, $b_N(x,y)=1$ if $|x-y|=1$ and $b_N(x,y)=0$
otherwise and $c\equiv 0$. Moreover, let the measure $m$ on $N$ be constant $m\equiv 1$.
The next lemma shows that when $u$ solves
$(\ow L_N + \al)u(x)=0$ for some $\al >0 $ and all $x\in
N\setminus\{0\}$ then it is only bounded if it is exponentially
decreasing.

\medskip

\begin{lemma}\label{l:N} Let   $(N,b_N,0)$  be as above and  $m\equiv1$. Let $\alpha >0$ be given.  Let
$u$ be a  positive function on $N$ with  $(\ow L_N+\al)u(x)=0$ for all  $x\geq1$. If for some $x\geq 1$
$$u(x)\geq\frac{2}{2+\al} u(x-1),$$
then $u$ increases exponentially.
\begin{proof} Let $u$ be a positive solution.
If $(1+\frac{\alpha}{2})u(x)\geq u(x-1)$ for some $x\geq 1$, we get by the
equation $(\ow L_N+\al)u(x)=0$
\begin{eqnarray*}
0&=&(1+\frac{\alpha}{2})u(x)-u(x+1)+(1+\frac{\alpha}{2})u(x)-u(x-1)\\
&\geq&(1+\frac{\alpha}{2})u(x)-u(x+1).
\end{eqnarray*}
This implies $u(x+1)\geq(1 + \frac{\alpha}{2}) u(x)$ and, in particular, $(1 + \frac{\alpha}{2}) u(x+1)\geq u(x)$. By induction we then get for $y\geq x$
$$u(y)\geq(1+\frac{\alpha}{2})^{y-x} u(x)$$ which gives the statement.
\end{proof}
\end{lemma}

\begin{proof}[Proof of Theorem~\ref{main1}.]
Let $(W,b_W,c_W)$ be a weighted graph and $m$ a measure of full support on $W$.
If $\SC$ holds we are done, so we assume the contrary. We will construct a weighted graph $(V,b,c)$ satisfying $\SC$ such that $W\subseteq V$ and $b\vert_{W\times W}=b_W$.
Define
$$\deg_{b_W} (x)= \frac{1}{m(x)}\sum_{y\in W}b(x,y).$$
 Let
$n:W\ra (0,\infty) $ be a function which satisfies $n(x)\deg_{b_W}(x)m(x)\in\N$
and
$$ \sum_{j=1}^\infty  n(x_j) =\infty$$
for any sequence $(x_j)$ in $W$. (For example we can set $ n(x) = {[\deg_{b_W} m +1]  }/{ \deg_{b_W} m   }$, where $[x]$ denotes the smallest integer not exceeding $x$.)

To each vertex $x\in W$, we attach  $n(x)\deg_{b_W}(x)m(x)$ copies of
the weighted  graph $(N,b_N,0)$ defined in the beginning of the section. We do
this by identifying $x\in W$ with the vertices $0$ in the associated
copies of $N$. We denote the resulting graph by $V$ and define $b$ on $V\times V$ by letting
$$b(x,y)=\left\{
\begin{array}{l@{\quad:\quad}l}
  b_W(x,y)& x,y\in W, \\
  b_N(x,y)& x,y\;\mathrm{in\;the\;same\;copy\;of\;}N,\\
  0&\mathrm{otherwise}.
\end{array}\right.$$
Moreover, we extend $c$ and $m$ to $V$ by letting $c\equiv0$ and $m\equiv1$ on $V\setminus W$
and denote $\ow L=\ow L_V$.
We will show that for all $\al>0$ every non-negative nontrivial function $u$ on
$V$, which satisfies $(\ow L+\al)u=0$,  is unbounded.  Without loss of generality, we can assume that the graph is connected. Then, any  non-negative nontrivial solution $u$ of   $(\ow L+\al)u=0$ must be positive by the minimum principle.  Let $u$ be  such a  positive solution of
$(\ow L+\al)u=0$ and assume it is bounded.

Fix  $x_0\in W$ and a  sequence $(\rho_r)$ in $\R$ with $(2+\alpha)/2 >\rho_r >1 $ and $\sum (\rho_r -1) <\infty$.  By induction we can now define  for each $r\in \N$ an  $x_r\in V$  such that
$b(x_r, x_{r-1})> 0$ and $\rho_r u(x_{r+1})\geq \sup_{y\in V, b(x_r,y)>0} u(y)$.
Since we assumed $u$ bounded,  Lemma \ref{l:N} gives $u(y) < 2u(x_r)/(2+\al) $
for each vertex $y$ in a copy of $N$ which is adjacent to $x_r$. If $x_{r+1}$ was in a copy of $N$, then this would imply that $u$ has a maximum in $x_r$ which leads to a contradiction to $(\widetilde L+\alpha)u=0$.
Thus all $x_r$ belong to $W$. The equation  $(\ow L+\al)u(x_r)=0$ now  gives
\begin{eqnarray*}
0&=& \frac{1}{m(x_r)} \sum_{y\in V}  b(x_r,y) (u(x_r) - u(y)) +  \frac{c(x_r)} {m(x_r)} u(x_r) + \alpha u(x_r)\\
&\geq&\deg_{b_W}(x_r)u(x_r)+\frac{1}{m(x_r)}\ab{\sum_{y\in
V\setminus W}b(x_r,y)(u(x_r)-u(y))-\sum_{y\in W}b(x_r,y)u(y)}\\
&\geq&\ab{1+\frac{\alpha  n(x_r)}{2+\al}}\deg_{b_W}(x_r)u(x_r)
-\rho_r\deg_{b_W}(x_r) u(x_{r+1}).
\end{eqnarray*}
In the second  inequality, we used $\al, c(x_r), u(x_r)\geq0$. In the third inequality, we estimated the sum over $y\in V\setminus W$ by the inequality $u(y)<2/(2+\al) u(x_{r})$ of
Lemma \ref{l:N} and the sum over $y\in W$ by the
choice of $x_{r+1}$. We get by direct calculation and iteration
$$u(x_{r+1})\geq
\frac{1}{\rho_{r}}\ab{\frac{\alpha  n(x_r)}{2+\al}+1}u(x_r)
\geq\ab{\prod_{j=1}^r\frac{1}{\rho_{j}}}\ab{\prod_{j=1}^r \ab{\frac{\alpha n(x_j)}{2+\al}+1}}u(x_0).$$
Letting $r$ tend to infinity the   right hand
side diverges if and only if $n$ is chosen such that
$\sum_{j=1}^\infty n(x_j)$ is divergent. (Notice that the infinite product over $(1/\rho_j)$ is greater than zero since we assumed that $(\rho_j-1)$ is summable).
Thus, by our choice of  $n$, we arrive at the contradiction  that $u$ is unbounded. By Theorem~\ref{main0}, this construction
shows that for every  $(W,b_W,c_W)$ there is a weighted graph $(V,b,c)$ which is $\SC$ and $(W,b_W,c_W)$ is a subgraph of $(V,b,c)$.
\end{proof}

\begin{remark}\rm{
An alternative construction is to add single vertices instead
of copies of $N$. For the resulting graph and a function $u$
satisfying $(\ow L+\al)u=0$ the value of $u$ on an added vertex
$y$ adjacent to the vertex  $x$ in the original graph is then determined by
$(1+\al)u(y)=u(x)$. The rest of the proof can now be carried out in a  similar manner. We chose
to do the construction above to avoid the impression that the
$\SC$ is the result of adding some type of
boundary to the graph.}
\end{remark}

We finish this section by  proving  Theorem~\ref{main1b}.
\begin{proof}[Proof of Theorem~\ref{main1b}] Set $b(x):= \sum_{y\in V} b(x,y)$.
Choose $c':V\to[0,\infty)$ such that for any sequence $(x_j)$ in  $V$ satisfying $b(x_j,x_{j+1})>0$ for all $j\in \N$, we have $$\sum_{j=1}^\infty \frac{c(x_j)+c'(x_j)}{b(x_j)}=\infty.$$
(For example one may choose $c'(x)=b(x)$ for $x\in V$.) We now follow
a similar reasoning as in the proof of Theorem~\ref{main1}: We consider  nontrivial non-negative solution $u$  of $(\ow L_{b,c+c',m}+\al)u=0$, $\al> 0$ and choose  inductively for each $r\in \NN$ an  $x_r\in V$ and $ 2 \geq \rho_r > 1$ with  $u(x_0) > 0$, $b(x_r,x_{r+1})>0$ and  $\rho_r u(x_{r+1}) \geq \sup_{ y: b(x_r,y)>0} u(y)$ for all $r\in \NN$. Then, a direct calculation gives
$ u(x_{r+1}) \geq \frac{1}{\rho_r} (1 + \frac{ c(x_{r}) + c'(x_{r})}{b(x)}) u (x_r)$
and unboundedness of $u$ follows (whenever $\rho_r$ converges to $1$ sufficiently fast). Hence,
by Theorem~\ref{main0} the graph $(V,b,c+c')$ satisfies $\SC$.
\end{proof}

\section{An incompleteness criterion} \label{An}
In this section we prove Theorem~\ref{main2}, which is a
counterpart to Theorem~\ref{main1}. As shown  there a subgraph with  $\SI$ is well compatible with the whole graph satisfying  $\SC$. Theorem~\ref{main2} shows under which
additional condition $\SI$ of a subgraph
implies $\SI$ for the whole graph. This condition
is about how heavily the incomplete subgraph is connected with the
rest of the graph. Not having control over the amount of connections
leads possibly to $\SC$ as we have seen in Theorem~\ref{main1}.

\smallskip

For a subset  $W$ of a weighted graph $(V,b,c)$ we
define the outer boundary $\dd W$  of $W$ in $V$ by
$$\dd W =\{x\in V\setminus W : \;\exists y\in W,\;b(x,y)>0\}.$$
Note that the outer boundary of $W$ is a subset of $V\setminus W$.  We will be concerned with decompositions of the whole set $V$ into two sets $W$ and $W':=V\setminus W$. In this case, there are two outer boundaries. Our intention is to extend  positive bounded functions $u$ on $W$ with  $(\widetilde{L}_W^{(D)} + \alpha) u  \leq 0$ to positive bounded  functions $v$  on the whole space satisfying $(\widetilde{L} + \alpha) v \leq 0$  . To do so, we will have to take particular care at  what happens on the two boundaries.

\begin{lemma}\label{hilfeeins} Let $(V,b,c)$ be a connected weighted graph.  Let $W\subseteq V$ be non-empty.  Then, any connected component of $\ov W= V\setminus W$ contains a point $x\in \dd W$.
\end{lemma}
\begin{proof} Choose $x\in W$ arbitrarily. By assumption, any $y\in \ov W$ is connected to $x$ by a path in $V$, i.e., there exist   $x_0,x_1,\ldots,x_n\in V$ with $b(x_i, x_{i+1}) >0$ and $x_0 = x$, $x_n = y$. Let $m\in \{0,\ldots, n\}$ be the largest number with $x_m\in W$. Then, $x_{m+1}$ belongs to both the boundary of $W$ and the connected component of $y$.
\end{proof}

\begin{lemma}\label{hilfezwei}  Let  $(V,b,c)$ be a weighted graph and $m$ a measure of full support. Let  $U\subseteq V$ and $\varphi$ be a non-negative function in $\ell^2 (U,m)$. Then, $(L_U^{(D)} + \alpha)^{-1} \varphi$ is  non-negative on $U$ and positive on  the connected components of any $x\in U$ with $\varphi (x) >0$.
\end{lemma}
\begin{proof} The operator $L_U^{(D)}$ is associated to the weighted graph  $(U,b_U^{(D)}, c_U^{(D)})$. Hence, Corollary \ref{positivityimproving} gives the statement.
\end{proof}

\begin{lemma}\label{hilfedrei} Let $(V,b,c)$ be a weighted graph and $m$ a measure of full support. Let $U\subseteq V$ be given. Let $v\in \widetilde{F}$ and denote the restriction of $v$ to $U$ by $u$ and the restriction of $v$ to $V\setminus U$ by $u'$.  Then, for any $x\in U$
 $$ (\widetilde{L} + \alpha) v (x) =  (\widetilde{L}_U^{(D)} + \alpha) u(x) - \frac{1}{m(x)}\sum_{y\in V\setminus U} b(x,y) u'(y).$$
\end{lemma}
\begin{proof} This follows by direct calculation.
\end{proof}

\begin{proof} [Proof of Theorem~\ref{main2}.]
Let  $W\subseteq V$ be given  such that for every $\al>0$ there is
a bounded non-negative nontrivial  function $u$ on $W$ satisfying
$$ (\widetilde{L}_W^{(D)}  +\al)u\leq0.$$
By Theorem \ref{main0},  it   suffices to show that any such  $u$ can
be extended to a non-negative and bounded function $v$ on $V$ such that
$$( \widetilde{L} +\al)v\leq 0.$$
To do so, we proceed as follows: Set $\ov W=V\setminus W$.   Define
$$\psi: \ov W \longrightarrow \RR,\; \: \psi (x) =  \frac{1}{m(x)}\sum_{y\in W} b(x,y) u (y).$$
Thus, $\psi$ vanishes on $\ov W\setminus \dd W$ and is non-negative  on $\dd W$.
Now, choose $\ph\in \ell^2 (\ov W,m_{W})$ with
$0\leq \ph\leq \psi$
and $ \ph (x)\neq 0$ whenever $\psi (x)\neq 0$.   Thus,
$$\ph\geq 0 \;\:\mbox{on}\:\; \dd W\;\:\mbox{and}\:\; \ph \equiv 0 \;\:\mbox{on}\:\;   \ov W\setminus \dd W.$$
Define $\ov u$ on $\ov W$ by
$$ \ov u:=(L_{\ov W}^{(D)} +\alpha)^{-1} \ph.$$
As $\ph$ is non-negative  on $\dd W$, combining Lemma \ref{hilfeeins} and Lemma \ref{hilfezwei} shows  that $\ov u$ is non-negative (on $\ov W$). Now,  define $v$ on $V$ by setting $v$ equal to $u$ on $W$ and setting $v$ equal to $\ov u$ on $\ov W$.
We now investigate  for each $x\in V$ the value of
$$(\widetilde{L} + \alpha) v (x).$$
We consider four cases.

\textit{Case 1: $x\in W \setminus \dd \ov W$}. Then, $(\widetilde{L} + \alpha) v (x) = (\widetilde{L}_W^{(D)}  +\al)u (x) \leq 0$ by assumption on~$u$.\\
\textit{Case 2: $x\in \ov W\setminus \dd W$}. Then, $(\widetilde{L} + \alpha) v (x) = (\widetilde{L}_{\ov W}^{(D)}  +\al)\ov u (x) = \ph (x)=0$ by construction of $\ov u$.\\
\textit{Case 3: $ x\in \dd \ov W$}.  Lemma \ref{hilfedrei} with $U =W$  gives
$$ (\widetilde{L} + \alpha) v (x) = (\widetilde{L}_{W}^{(D)} +\alpha) u(x) - \sum_{y\in \ov W} b(x,y) \ov u (y) \leq 0.$$
Here, the last inequality follows as $(\widetilde{L}_{W}^{(D)} +\alpha) u(x)\leq 0$ by assumption on $u$ and
$\ov u$ is non-negative.\\
\textit{Case 4: $ x\in \dd W$}.  Lemma \ref{hilfedrei} with $U =\ov W$ gives
\begin{align*}
 (\widetilde{L} + \alpha) v (x) = (\widetilde{L}_{\ov W}^{(D)} +\alpha) \ov u(x) - \sum_{y\in W} b(x,y) u (y)  = \ph (x) - \psi (x) \leq 0.
\end{align*}
This finishes the proof.
\end{proof}

\begin{proof} [Proof of Corollary~\ref{beispiel}.] Let $C\geq0$ be a constant such that $\sum_{y\in V\setminus W} b(x,y)/m(x)\leq C$ for all $x\in W$ and let $\alpha>0$. A non-negative subsolution for $\al+C$ with respect to the operator associated to $(b_W^{(D)},0)$ is obviously a subsolution for $\al$ with respect to the operator associated to $(b_W^{(D)},c_W^{(D)})$. By assumption such non-negative, nontrivial,  bounded subsolutions for $\al+C$ and $(b_W^{(D)},0)$ exist for all $\al>0$. Therefore $(b_W^{(D)},c_W^{(D)})$ satisfies $\SI$. The statement now follows from Theorem~\ref{main2}.
\end{proof}

\footnotesize{
\textbf{Acknowledgements.} The research of M.K. is financially supported by a grant from Klaus Murmann Fellowship Programme (sdw). Part of this work was done while he was visiting Princeton University. He would like to thank the Department of Mathematics for its hospitality. He would also like to thank Jozef Dodziuk and Radek Wojciechowski for several inspiring discussions bringing up some of the questions which motivated this paper. D.L. would like to thank Andreas Weber for most stimulating discussions and Peter Stollmann for generously sharing his knowledge on Dirichlet forms on many occasions. Partial support from German Science Foundation (DFG) is gratefully acknowledged.
}


\begin{thebibliography}{10}

\bibitem{BD}   A. Beurling, J. Deny,  Espaces de Dirichlet. I. Le cas \'{e}l\'{e}mentaire, Acta Math.  \textbf{99}  (1958),  203--224.

\bibitem{BD2}  A. Beurling, J. Deny,  Dirichlet spaces,
 Proc. Nat. Acad. Sci. U.S.A.  \textbf{45} (1959),  208--215.


\bibitem{BH}
N.~Bouleau, F.~Hirsch,  { Dirichlet forms and analysis on {W}iener space},  de Gruyter Studies in Mathematics, \textbf{14}, {de Gruyter} , 1991.




\bibitem{Chu} F. R. K. Chung,  Spectral Graph Theory, CBMS Regional Conference Series in Mathematics \textbf{92},  Am. Math. Soc, 1997.

\bibitem{Col} Y. Colin de Verdi\`{e}re,  Spectres de graphes, 4. Soc. Math.  France, Paris,  1998.

\bibitem{Dav1} E. B. Davies, Heat kernels and spectral theory, Cambridge
  University press, Cambridge 1989.


\bibitem{Dav3} E. B. Davies, Linear operators and their spectra.
Cambridge Studies in Advanced Mathematics, 106. Cambridge University Press, Cambridge,  2007.

\bibitem{Dod} J. Dodziuk, Difference Equations, Isoperimetric Inequality and Transience of Certain Random Walks, {Trans. Amer. Math. Soc.},\textbf{284}, (1984), 787--794.

\bibitem{Dod2} J.   Dodziuk, Elliptic operators on infinite graphs,   Analysis, geometry and topology of elliptic operators,  353--368, World Sci. Publ., Hackensack, NJ, 2006.


\bibitem{DK} J. Dodziuk, W. S. Kendall, Combinatorial Laplacians and Isoperimetric Inequality, From Local Times to Global Geometry, Control and Physics, Pitman Research Notes in Mathematics, \textbf{150} (1986), 68--74.

\bibitem{DM} J. Dodziuk, V. Matthai, Kato's inequality and asymptotic spectral properties for discrete magnetic Laplacians.  The ubiquitous heat kernel,   Cont. Math. \textbf{398}, Am. Math. Soc. (2006), 69--81.




\bibitem{Fel} W. Feller,  On boundaries and lateral conditions for the Kolmogorov differential equations.  Ann. of Math. (2)  \textbf{65}  (1957), 527--570

\bibitem{Fel2} W. Feller, Notes to my paper ``On boundaries and lateral conditions for the Kolmogorov differential equations.'' Ann. of Math. (2) \textbf{68} (1958), 735--736.


\bibitem{Fuk} M. Fukushima, Y. Oshima, M. Takeda,  Dirichlet forms and symmetric Markov processes, de Gruyter Studies in Mathematics, \textbf{19}, de Gruyter, 1994.

\bibitem{Gri} A. Grigor'yan, Analytic and Geometric Background of Reccurrence and Non-Explosion of the Brownian Motion on Riemannian Manifolds, {Bull. Amer. Math. Soc.} \textbf{36}, (1999), 135--249.

\bibitem{HK} S. Haeseler, M. Keller, Generalized solutions and spectrum for Dirichlet forms on graphs, to appear in:  D. Lenz, F. Sobieczky, W. Woess (eds): {Boundaries and Spectra of Random Walks}, Progress in Probability, Birkhaeuser.

\bibitem{Jor} P. E. T. Jorgensen, Essential selfadjointness of the graph-Laplacian,  J. Math. Phys.  49  (2008).

\bibitem{Kel} M. Keller, \emph{Essential Spectrum of the Laplacian on rapidly branching tessellations.}, Math. Ann. \textbf{346}, Issue 1 (2010), 51--66.

\bibitem{KL} M. Keller, D. Lenz, Unbounded Laplacians on graphs: Basic spectral properties and the heat equation, Math. Model. Nat. Phenom. \textbf{5}, No. 2, 2010.

\bibitem{MR}
Z. M.~Ma,  M.~R{\"o}ckner, {{Introduction to the theory of (non-symmetric) Dirichlet   forms}}, {Springer}, 1992.


\bibitem{MS} B. Metzger, P. Stollmann,  Heat kernel estimates on weighted  graphs,  Bull. London Math. Soc.  \textbf{32}  (2000),  477--483.

\bibitem{Reu} G. E. H. Reuter, Denumerable Markov processes and the associated contraction semigroups on $l$. Acta Math. \textbf{97} (1957),  1--46.

\bibitem{Stu}
K.-T. Sturm.  {Analysis on local Dirichlet spaces. I: Recurrence, conservativeness  and $L\sp p$-Liouville properties.}
{J. Reine Angew. Math.}, \textbf{456}, (1994) 173--196.

\bibitem{Sto} P. Stollmann, A convergence theorem for Dirichlet forms with applications to boundary  value problems with varying domains.  Math. Z.  \textbf{219},   (1995),  275--287.

\bibitem{SV} P. Stollmann, J. Voigt, Perturbation of Dirichlet forms by measures.  Pot. Anal.  \textbf{5}  (1996), 109--138.


\bibitem{Web} A. Weber, Analysis of the physical Laplacian and the heat flow on
  a locally finite graph, J. Math. Anal. Appl. \textbf{370} (2010), 146--158.

\bibitem{Wei} J. Weidmann, Lineare Operatoren in Hilbertr\"aumen I,  Teubner, Stuttgart, 2000


\bibitem{Woj} R. K. Wojciechowski, Stochastic completeness of graphs,  PHD thesis, (2007), arXiv:0712.1570v2.

\bibitem{Woj2} R. K. Wojciechowski, Heat kernel and
essential spectrum of infinite graphs,  Indiana Univ. Math. J.  \textbf{58}  (2009),   1419--1441.,  arXiv:0802.2745.

\end{thebibliography}
\end{document}